\documentclass[12pt]{amsart}
\usepackage[utf8]{inputenc}
\usepackage{amsmath,amssymb,amsthm,colonequals,mathrsfs,mathtools}
\usepackage{array,enumitem,yfonts}
\usepackage{comment}
\usepackage[alphabetic]{amsrefs}
\usepackage[all,cmtip]{xy}
\usepackage[colorlinks,anchorcolor=blue,citecolor=blue,linkcolor=blue,urlcolor=blue,bookmarksopen=true]{hyperref}
\allowdisplaybreaks

\usepackage[margin=1in]{geometry}

\usepackage{tikz}
\usetikzlibrary{cd,arrows}
\tikzset{>=stealth}
\tikzcdset{arrow style=tikz}
\tikzset{link/.style={column sep=1.8cm,row sep=0.16cm}}

\AtBeginDocument{%
	\def\MR#1{}
}
\usepackage{fancyhdr}
\fancypagestyle{titlepage}
{
	\fancyhf{}

	\fancyfoot[l]{
	\href{https://mathscinet.ams.org/mathscinet/msc/msc2020.html}
	}
}

\newcommand{\bF}{\mathbb{F}}    

\newcommand{\bP}{\mathbb{P}}    

\newcommand{\kchar}{\operatorname{char}} 
\newcommand{\PGL}{\operatorname{PGL}} 
\newcommand{\Spec}{\operatorname{Spec}}

\newtheorem{thm}{Theorem}[section]
\newtheorem{prop}[thm]{Proposition}
\newtheorem{lemma}[thm]{Lemma}
\newtheorem{cor}[thm]{Corollary}
\numberwithin{equation}{section}

\theoremstyle{definition}
\newtheorem{conj}[thm]{Conjecture}
\newtheorem{defn}[thm]{Definition}
\newtheorem{eg}[thm]{Example}

\newtheorem{rmk}[thm]{Remark}

\begin{document}

\title{Frobenius nonclassical hypersurfaces}
\author{Shamil~Asgarli
	\and Lian~Duan
	\and Kuan-Wen~Lai}

\newcommand{\ContactInfo}{{
\bigskip\footnotesize

\bigskip
\noindent S.~Asgarli,
\textsc{Department of Mathematics and Computer Science \\
Santa Clara University \\ 
Santa Clara, CA 95050, USA}\par\nopagebreak
\noindent\texttt{sasgarli@scu.edu}

\bigskip
\noindent L.~Duan,
\textsc{Institute of Mathematical Sciences\\
ShanghaiTech University\\
No.393 Middle Huaxia Road, Pudong New District, 
Shanghai, China}\par\nopagebreak
\noindent\textsc{Email:} \texttt{duanlian@shanghaitech.edu.cn}

\bigskip
\noindent K.-W.~Lai,
\textsc{Department of Smart Computing and Applied Mathematics \\
Tunghai University \\
No.~1727, Sec.~4, Taiwan~Blvd., Xitun~Dist. \\
Taichung~City 407224, Taiwan}\par\nopagebreak
\noindent\textsc{Email:} \texttt{kwlai@thu.edu.tw}
}}

\begin{abstract}
A smooth hypersurface over a finite field $\mathbb{F}_q$ is called Frobenius nonclassical if the image of every geometric point under the $q$-th Frobenius endomorphism remains in the unique hyperplane tangent to the point. In this paper, we establish sharp lower and upper bounds for the degrees of such hypersurfaces, give characterizations for those achieving the maximal degrees, and prove in the surface case that they are Hermitian when their degrees attain the minimum. We also prove that the set of $\mathbb{F}_q$-rational points on a Frobenius nonclassical hypersurface form a blocking set with respect to lines, which indicates the existence of many $\mathbb{F}_q$-points.
\end{abstract}

\maketitle
\thispagestyle{titlepage}

\setcounter{tocdepth}{2}
\makeatletter
\def\l@subsection{\@tocline{2}{0pt}{2.3pc}{5pc}{}}
\makeatother

\tableofcontents

\section{Introduction}
\label{sect:intro}

In the celebrated paper \cite{SV86}, St\"ohr and Voloch used Weierstrass order-sequences to obtain an upper bound on the number of rational points for curves embedded in an arbitrary projective space over finite fields. The concept of Frobenius nonclassical curves was introduced naturally in their work as those curves whose order sequence behaves differently from ``most'' curves. Afterward, Hefez and Voloch \cite{HV90} extensively studied the properties of Frobenius nonclassical curves. In addition to showing that such curves are always nonreflexive, they also computed the precise number of $\bF_q$-points on Frobenius nonclassical \emph{plane} curves. It turns out that Frobenius nonclassical plane curves have many $\bF_q$-points. Inspired by this observation, points on Frobenius nonclassical plane curves have been used to construct new complete arcs in the field of combinatorial geometry \cites{GPTU02, Borges-Arcs, BMT14}. 

As a natural generalization of the plane curve case, our earlier work \cite{ADL21} introduced the concept of Frobenius nonclassical hypersurfaces, with a particular emphasis on the surface case. We found that if a hypersurface is \emph{not} Frobenius nonclassical, then it is easier to prove a Bertini-type theorem such as the existence of a smooth hyperplane section over the ground field \cite{ADL21}*{Theorem~2.2}, \cite{ADL24hyper}*{Theorem~1.4}. The purpose of the present paper is to provide a systematic study of Frobenius nonclassical hypersurfaces in arbitrary dimensions. We extend some of the results known for plane curves, such as the lower and upper bounds on the degree to higher dimensions, and give evidence for the abundance of $\mathbb{F}_q$-points on these hypersurfaces.

\begin{defn}
\label{defn:FrobNoncal}
Suppose that $X\subset\bP^n$ is a hypersurface defined by a polynomial $F$ over a finite field~$\bF_q$.  We say that $X$ is \emph{Frobenius nonclassical} if $F$ divides the polynomial
\footnote{
    The notation $F_{1,0}$ originates from our earlier work \cite{ADL21}, where we consider more generally the polynomial $F_{a,b} \colonequals \sum_i x_i^{q^a}(\partial F/\partial x_i)^{q^b}$.
}
$$
    F_{1,0}\colonequals 
    \sum_{i=0}^n x_i^q \frac{\partial F}{\partial x_i}.
$$
\end{defn}

Under this definition, every $\bF_q$-hyperplane is Frobenius nonclassical over $\bF_q$. Below, we present some examples that are not hyperplanes.

\begin{eg}
\label{eg:d=q+1_n-odd}
For $n = 2m+1$, the smooth hypersurface $X\subset\bP^n$ defined by
$$
    F = \sum_{i=0}^{m}\left(
        x_{2i}^qx_{2i+1} - x_{2i}x_{2i+1}^q
    \right)
$$
is Frobenius nonclassical over $\bF_q$ since $F_{1,0} = 0$. The hypersurface $X$ has several remarkable properties. For example, it is \emph{space-filling}, namely, it passes through every $\mathbb{F}_q$-point of the ambient $\bP^n$. On the other hand, if $H\subset \bP^n$ is any two-dimensional plane defined over $\mathbb{F}_q$, then either $H\subset X$, or $X\cap H$ is a union of $q+1$ distinct $\bF_q$-lines passing through a common point. This second property was carefully examined in \cite{ADL21}*{Example 3.4} in the special case when $n=3$.
\end{eg}

\begin{eg}[Hermitian varieties]
\label{eg:hermitian}
For a \emph{square $q$}, the hypersurface $X\subset\bP^n$ defined by
\begin{equation}
\label{eqn:std_Hermitian}
    F = \sum_{i=0}^rx_i^{\sqrt{q} + 1}
    \qquad\text{where}\qquad
    0\leq r\leq n
\end{equation}
is Frobenius nonclassical over $\bF_q$ as one can verify that $F_{1,0} = (F)^{\sqrt{q}}$. This is an example of a \emph{Hermitian variety} defined as follows: Consider the involution $x\mapsto x^{\sqrt{q}} = \overline{x}$ on $\bF_q$ and, for a scalar matrix $H = (h_{ij})$, write $\overline{H} = (\overline{h_{ij}})$. A Hermitian variety is a hypersurface in $\bP^n$ defined by a polynomial of the form
$$
    F = \overline{\mathbf{x}}\cdot H\cdot\mathbf{x}^t
$$
where $\mathbf{x} = (x_0,\dots,x_n)$ and $H$ is a scalar matrix that satisfies $\overline{H}^t = H\neq 0$. In fact, every such $F$ is projectively equivalent to \eqref{eqn:std_Hermitian} over $\bF_q$ by \cite{BC66}*{Theorem~4.1}.
\end{eg}

Before introducing the main results, let us set up some terminology to be used throughout the paper. Consider a scheme $X$ defined over $\mathbb{F}_q$.
\begin{itemize}
    \item $X$ is \emph{reduced} if $\mathcal{O}_X(U)$ has no nonzero nilpotents for every open subset $U\subset X$.
    \item $X$ is \emph{irreducible} if its underlying topological space is irreducible, and is \emph{geometrically irreducible} if the base change
    $
        X_{\overline{\mathbb{F}_q}}\colonequals
        X\otimes_{\Spec(\mathbb{F}_q)} \Spec(\overline{\mathbb{F}_q})
    $
    is irreducible.
    \item $X$ is \emph{geometrically integral} if it is reduced and geometrically irreducible.
\end{itemize}
Because $\bF_q$ is perfect, reducedness over the ground field $\bF_q$ and the algebraic closure $\overline{\bF_q}$ are equivalent, which is why we do not need to introduce the notion of ``geometric reducedness''. Suppose that $X=\{F=0\}\subset\mathbb{P}^{n}$ is a hypersurface defined over $\bF_q$, with the factorization of $F$ over $\overline{\bF_q}$ being
$$
    F = \prod_{i=1}^{r} T_i
    \;\in\;\overline{\bF_q}[x_0,\dots,x_n].
$$
In this situation, the above properties are the same as follows:
\begin{itemize}
    \item $X$ is reduced if and only if all the factors $T_i$ are mutually coprime.
    \item $X$ is geometrically irreducible if and only if all the factors $T_i$ coincide up to rescaling.
    \item $X$ is geometrically integral if and only if $r=1$.
\end{itemize}

\begin{rmk}
\label{rmk:geoFrobNoncal}
Suppose that $X\subset\bP^n$ is a \emph{reduced} hypersurface defined by a polynomial~$F$ over~$\bF_q$. Then $X$ is Frobenius nonclassical if and only if the $q$-th Frobenius morphism
$$
    \Phi\colon\bP^n\longrightarrow\bP^n
    : [x_0:\dots:x_n]\longmapsto[x_0^q:\dots:x_n^q]
$$
maps every $\overline{\bF_q}$-point $P\in X$ into the embedded tangent hyperplane
$$
    T_PX\colonequals\left\{
        \sum_{i=0}^n x_i\left(
            \frac{\partial F}{\partial x_i}(P)
        \right) = 0
    \right\}\subset\bP^n.
$$
Therefore, one may take this property as the definition of a Frobenius nonclassical hypersurface when $X$ is reduced.
\end{rmk}

Our first main result concerns lower bounds on the degree of a Frobenius nonclassical hypersurface. The result below was proved in \cite{HV90}*{Proposition~6} for smooth plane curves, and in \cite{BH17}*{Corollary~3.2} when the curves are irreducible.

\begin{thm}
\label{mainthm:lower-bound}
Let $X\subset\bP^n$, where $n\geq 2$, be a Frobenius nonclassical hypersurface of degree $d\geq 2$ over a finite field $\bF_q$ of characteristic $p$ which is smooth at all its $\bF_q$-points. Then 
\begin{itemize}
    \item $d\geq p+1$, and
    \item $d\geq\sqrt{q}+1$ provided that $X$ is reduced.
\end{itemize}
In the case that $X$ is a reduced curve or a smooth surface, the condition $d = \sqrt{q}+1$ is attained only if $X$ is Hermitian.
\end{thm}

The two lower bounds in Theorem~\ref{mainthm:lower-bound} are sharp due to Examples~\ref{eg:d=q+1_n-odd} and \ref{eg:hermitian}, respectively. Note that the bound $d\geq p+1$ is not redundant with respect to $d\geq\sqrt{q}+1$ as the former is stronger than the latter when $q = p$. The reducedness for the second bound is necessary because, if the assumption is dropped, then one can easily construct counterexamples using $\bF_q$-hyperplanes and $p$-th powers of pointless hypersurfaces; see Remark~\ref{rmk:pointless_example} for concrete examples of pointless hypersurfaces. We will prove the two bounds respectively in Theorems~\ref{thm:lower-bound_p+1} and \ref{thm:lower-bound_sqrt(q)+1}. For the last statement, the curve case is essentially established by Borges and Homma \cite{BH17}*{Corollary~3.2} as mentioned above, but we still need to treat the case when the curve is potentially reducible; see Lemma~\ref{lemma:min-degree-curve}. We treat the surface case in Proposition~\ref{prop:Hermitian-surf}. It is worth mentioning that Homma and Kim \cite{HK16} gave an alternative characterization of Hermitian surfaces via the number of rational points.

Our next main result concerns upper bounds on the degree.

\begin{thm}
\label{mainthm:upper-bound}
Suppose that $X\subset\bP^n$, where $n\geq 2$, is a smooth Frobenius nonclassical hypersurface of degree~$d$ over $\bF_q$. Then $d\leq q+2$ and, when $d=q+1$ or $d=q+2$, its defining polynomial~$F$ has the following forms:
\begin{enumerate}[label=\textup{(\arabic*)}]
\item\label{main:d=q+1_n-odd}
$d = q+1$ and $n$ is odd if and only if there exists a nondegenerate skew-symmetric matrix $(A_{ij})$ with entries in $\bF_q$ and zero diagonal such that
$$
    F = \sum_{i,j=0}^nx_i^qA_{ij}x_j.
$$
\item\label{main:d=q+1_n-even}
$d = q+1$ and $n$ is even if and only if $p = 2$ and, up to a $\PGL_{n+1}(\bF_q)$-action,
$$
    F_{1,0} = x_0^{q-1}F
    \qquad\text{and}\qquad
    F = x_0G
    + \sum_{i,j=1}^nx_i^qB_{ij}x_j,
$$
where we have $\partial G/\partial x_0 = 0$, and $(B_{ij})_{1\leq i,j\leq n}$ is a nondegenerate skew-symmetric matrix with entries in $\bF_q$ and zero diagonal.
\item\label{main:d=q+2}
$d = q+2$ if and only if $p = n = 2$ and, upon rescaling $F$ by a nonzero constant,
$$
    F = x_0x_1x_2(x_0^{q-1} + x_1^{q-1} + x_2^{q-1}) + G(x_0^2,x_1^2,x_2^2)
$$
for some polynomial $G$.
\end{enumerate}
Furthermore, condition \ref{main:d=q+1_n-odd} or \ref{main:d=q+2} occurs if and only if $F_{1,0}$ is the zero polynomial.
\end{thm}

Example~\ref{eg:d=q+1_n-odd} is a special case of Theorem~\ref{mainthm:upper-bound}~\ref{main:d=q+1_n-odd}. For concrete examples that are classified by \ref{main:d=q+1_n-even} and \ref{main:d=q+2} in the theorem, one can find them in Examples~\ref{eg:d=q+1_n-even} and \ref{eg:d=q+2}, respectively. As an immediate consequence of Theorems~\ref{mainthm:lower-bound} and \ref{mainthm:upper-bound}:

\begin{cor}
Let $X\subset\bP^n$, where $n\geq 2$, be a smooth Frobenius nonclassical hypersurface of degree~$d\geq 2$ over a finite field~$\bF_q$ of characteristic~$p$. Then
$$
    \max\{p+1, \sqrt{q}+1\}\leq d\leq
    \begin{cases}
        q+1 & \text{if}\quad p\text{ is odd}, \\
        q+2 & \text{if}\quad p=2.
    \end{cases}
$$
In particular, if $p$ is odd and $\bF_q$ is a prime field (namely $q=p$), then $d = p+1$.
\end{cor}

Examples characterized by Theorem~\ref{mainthm:upper-bound}~\ref{main:d=q+1_n-odd} are \emph{multi-Frobenius nonclassical} in the following sense: Let $X = \{F = 0\}\subset\bP^n$ be such an example. If we consider $X$ as a variety over the quadratic extension $\bF_{q^2}$, then
$$
    F_{1,0}\colonequals
    \sum_{i=0}^nx_i^{q^2}\frac{\partial F}{\partial x_i} = \sum_{i,j=0}^n x_i^qA_{ij}x_j^{q^2} = -F^q.
$$
This shows that $X$ is Frobenius nonclassical over not only $\bF_q$ but also $\bF_{q^2}$. In fact, if we pick an $\alpha\in\bF_{q^2}\setminus\{0\}$ that satisfies $\alpha^q = -\alpha$, then $X = \{\alpha F = 0\}$ is a Hermitian variety over $\bF_{q^2}$. The classification of multi-Frobenius nonclassical plane curves was carried out by Borges in ~\cite{Bor09}.

Notice that not every Frobenius nonclassical hypersurface is Hermitian. Example~\ref{eg:d=q+2}, which belongs to Theorem~\ref{mainthm:upper-bound}~\ref{main:d=q+2}, provides smooth such examples in characteristic~$2$. For smooth examples in arbitrary characteristics, one can consider
$$
    F = \sum_{i=0}^nx_i^{q^r+\dots+q+1}
    \qquad\text{with}\qquad
    r\geq 2
$$
and view it as a polynomial over $\bF_{q^{r+1}}$. Then
$
    F_{1,0} = \sum_{i=0}^nx_i^{q^{r+1}+\dots+q} = F^{q}.
$
Therefore, the hypersurface $X = \{F = 0\}\subset\bP^n$ is Frobenius nonclassical over $\bF_{q^{r+1}}$ and it is clearly not Hermitian over any ground field.

Hermitian varieties, which include examples in Theorem~\ref{mainthm:upper-bound}~\ref{main:d=q+1_n-odd} as discussed above, are defined by polynomials of the form
$$
    F = \sum_{i=0}^nx_i^{q'}L_i
$$
where $L_0,\dots,L_n$ are linear polynomials, and $q'$ is a power of the characteristic of the ground field. Such polynomials are called \emph{Frobenius forms} in \cite{KKPSSVW22}. In the last part of Section~\ref{sect:degree_lower_bound}, we show that Frobenius nonclassical hypersurfaces over $\bF_q$ of degree $\sqrt{q}+1$ are defined by Frobenius forms under certain assumptions, and then prove that they must be Hermitian in that situation. This result, together with Theorem~\ref{mainthm:lower-bound}, suggests the following conjecture:

\begin{conj}
\label{conj:hermitian}
A Frobenius nonclassical hypersurface $X\subset\bP^n$, where $n\geq 2$, over $\bF_q$ of degree $\sqrt{q}+1$ is Hermitian provided that it is reduced and smooth at all its $\bF_q$-points.
\end{conj}

Smooth Frobenius nonclassical hypersurfaces are nonreflexive \cite{ADL21}*{Theorem~4.5}, so their degrees are either congruent to $0$ or $1$ modulo the characteristic of the ground field \cite{Kle86}*{page 191}. For curves, the possibility of being congruent to $0$ has been excluded by Pardini \cite{Par86}*{Corollary~2.2}, so it is natural to ask if the same condition holds in higher dimensions. We prove that this is true with certain additional assumptions. To be more concrete, we say that a hypersurface $X\subset\bP^n$ over $\bF_q$ has \emph{separated variables} if, up to a projective transformation over $\bF_q$, its defining polynomial has the form
$$
	F(x_0, \ldots, x_n)=G(x_0, \ldots, x_m) + H(x_{m+1},\ldots, x_n)
$$
for some $m\in\{0,\dots, n-1\}$. Frobenius nonclassical components of plane curves with separated variables were studied by Borges in \cite{Bor16}.

Our final result establishes the congruence condition on the degree of a Frobenius nonclassical hypersurface with separated variables.

\begin{thm}
\label{mainthm:sep_var_d=1-mod-p}
Let $X = \{F=0\}\subset\bP^n$ be a smooth Frobenius nonclassical hypersurface of degree $d$ over $\bF_q$ with separated variables. Then $d\equiv 1 \pmod{p}$ where $p = \operatorname{char}(\bF_q)$.
\end{thm}

\vspace{7pt}

\noindent\textbf{Organization of the paper.} The present paper is organized as follows. In Section~\ref{sect:rational-point_univ-lower-bound}, we discuss the existence and multitude of $\bF_q$-rational points and prove the lower bound $d\geq p+1$. Section ~\ref{sect:degree_lower_bound} is devoted to the proof of the lower bound $d\geq \sqrt{q}+1$. In the same section, we also show that smooth Frobenius nonclassical surfaces of degree $\sqrt{q}+1$ are Hermitian and provide evidence about this phenomenon in higher dimensions. The proof of the upper bound $d\leq q+2$ is given in Section~\ref{sect:degree_upper_bound} along with the classification result in the cases $d=q+2$ and $d=q+1$. Finally, Section~\ref{sect:sep-var} is devoted to the study of Frobenius nonclassical hypersurfaces with separated variables and contains the proof of Theorem~ \ref{mainthm:sep_var_d=1-mod-p}.

It follows from Definition~\ref{defn:FrobNoncal} that a homogeneous polynomial $F\in\bF_q[x_0,\dots,x_n]$ defines a Frobenius nonclassical hypersurface over $\bF_q$ if and only if $F\cdot G^p$ does for any homogeneous polynomial $G\in\bF_q[x_0,\dots,x_n]$. In particular, a hypersurface $X = \{F=0\}\subset\bP^n$ is Frobenius nonclassical when $F$ itself is a $p$-th power. For most of the paper, we consider only hypersurfaces whose defining polynomials are free of $p$-th power factors.

In many proofs within the paper, we frequently use Euler's formula for homogeneous polynomials, which states that 
$$
    \sum_{i=0}^{n} x_i \frac{\partial F}{\partial x_i}
    = \deg(F) \cdot F
$$
for every homogeneous polynomial $F$ in variables $x_0, \dots, x_n$. Note that the definition of $F_{1,0}$ is independent of coordinate transformations over $\bF_q$, so the property of being Frobenius nonclassical is preserved under such a transformation. Similarly, the property of being Frobenius nonclassical is preserved under taking $\bF_q$-hyperplane sections, which is useful when carrying out a proof by induction on the dimension.

We conclude the introduction by noting that every $\bF_q$-irreducible component of a Frobenius nonclassical hypersurface, \emph{which is free of $p$-th power factors}, is still Frobenius nonclassical.

\begin{prop}
Suppose that $X=\{F=0\}\subset \mathbb{P}^n$ is a Frobenius nonclassical hypersurface over $\bF_q$ with $F = \prod_{j=1}^m F_j^{r_j}$, where $F_1,\ldots,F_m\in\bF_q[x_0, \ldots, x_n]$ are mutually distinct irreducible factors. Then, for each $1\leq j\leq m$, the component $X_j =\{F_j=0\}$ is Frobenius nonclassical provided that $p\nmid r_j$.
\end{prop}

\begin{proof}
First, we have that
\begin{align*}
    F_{1,0}
    = \sum_{i=0}^n x_i^q \frac{\partial F}{\partial x_i}
    &= \sum_{i=0}^n x_i^q\left(
        \sum_{j=1}^m F_1^{r_1}\cdots
        \frac{\partial F_j^{r_j}}{\partial x_i}
        \cdots F_m^{r_m}
    \right) \\
    &= \sum_{j=1}^m \left(
        r_j(F_j)_{1,0}F_j^{r_j-1}
        \prod_{k\neq j} F_k^{r_k}
    \right)
    = \left(
        \prod_{k=1}^{m} F_k^{r_k-1}
    \right)\sum_{j=1}^m \left(
        r_j(F_j)_{1,0}
        \prod_{k\neq j} F_k
    \right).
\end{align*}
By hypothesis, $F = \prod_{j=1}^m F_j^{r_j}$, and thus the factor $F_\ell^{r_\ell}$ for each $1\leq\ell\leq m$, divides the last expression above. This implies that
$$
    F_\ell\quad\text{divides}\quad
    \sum_{j=1}^m \left(
        r_j(F_j)_{1,0} \prod_{k\neq j} F_k
    \right).
$$
Note that $F_\ell$ is a factor of the $j$-th summand on the right-hand side for all $j\neq\ell$. Hence
$$
    F_\ell\quad\text{divides}\quad
    r_\ell(F_\ell)_{1,0} \prod_{k\neq\ell} F_k.
$$
If $p\nmid r_\ell$, then $F_\ell$ divides $(F_\ell)_{1,0}$ as it is prime to $\prod_{k\neq\ell} F_k$. This proves the statement.
\end{proof}

\vspace{7pt}

\noindent\textbf{Acknowledgements.}
We thank Nathan Kaplan and Felipe Voloch for their informative comments. We are also grateful to the anonymous referee for the excellent suggestions. The first author was partially supported by an NSERC PDF award and a postdoctoral research fellowship from the University of British Columbia. The initial draft of this paper was completed while the third author was supported by the ERC Synergy Grant HyperK (ID: 854361). The author is currently supported by the NSTC Research Grant (113-2115-M-029-003-MY3).

\section{Rational points and the universal lower bound on degree}
\label{sect:rational-point_univ-lower-bound}

In this section, we prove the universal lower bound $d\geq p+1$ in Theorem~\ref{mainthm:lower-bound}. One of the key ingredients is Lemma~\ref{lemma:incidence_non-pth-power} which asserts the existence of an $\bF_q$-point on a Frobenius nonclassical hypersurface of degree $d\leq q+1$. At the end of this section, we establish an explicit lower bound for the number of $\bF_q$-points in Proposition~\ref{prop:blocking} and, in the case of surfaces, a refined lower bound in Proposition~\ref{prop:lower-bound-rat-points-surfaces}.

\subsection{Existence of rational points}
\label{subsect:exist-rational-points}

Here we prove that, under mild assumptions, every Frobenius nonclassical hypersurface contains $\bF_q$-points. We first prove a Bertini-type result for hypersurfaces that are not a $p$-th power.

\begin{lemma}
\label{lemma:Bertini_not-pth-power}
Let $X=\{F=0\}\subset \bP^n$, where $n\geq 2$, be a hypersurface of degree $d$ over $\bF_q$ of characteristic $p$ such that $d\leq q+1$ and $X$ is not a $p$-th power. Then, for every $1\leq r\leq n-1$, there exists a linear subspace $H\subset\bP^n$ over $\bF_q$ of dimension $r$ such that $H\not\subset X$ and the restriction $F|_H$ is not a $p$-th power.
\end{lemma}

\begin{proof}
By induction, it suffices to prove the statement only for the case $r = n-1$, that is, for the case when $H$ is an $\bF_q$-hyperplane. Because $F$ is not a $p$-th power, there exists a system of homogeneous coordinates $\{x_0,\dots,x_n\}$ such that
$$
    F = x_0^sG_s(x_1,\dots,x_n) + \sum_{m\neq s}x_0^mG_m(x_1,\dots,x_n)
$$
where $s\not\equiv 0\pmod{p}$ and $G_s\neq 0$. The hyperplanes in $\bP^n$ of the form
$$
    H_{\underline{a}}
    \colonequals\{a_1x_1+\cdots+a_nx_n = 0\}
    \qquad\text{where}\qquad
    \underline{a} = (a_1,\dots,a_n)
    \in(\bF_q)^{n}\setminus\{0\}
$$
are parametrized by $\bP^{n-1}_{\bF_q}$. Thus there are $\frac{q^{n}-1}{q-1}$ many of them. If $G_s|_{H_{\underline{a}}} = 0$ for all $H_{\underline{a}}$, then the hypersurface $\{G_s = 0\}$ contains every $H_{\underline{a}}$ as a component, whence
$$
    \deg(G_s)\geq\frac{q^{n}-1}{q-1}.
$$
But this is impossible since
$$
    \deg(G_s) = d-s \leq q
    < \frac{q^{n}-1}{q-1}
$$
where the last inequality holds whenever $n\geq 2$. This shows that there exists $H_{\underline{a}}$ such that $G_s|_{H_{\underline{a}}}\neq 0$. Therefore, the coefficient of $x_0^s$ in $F|_{H_{\underline{a}}}$ is nontrivial, which implies that $F|_{H_{\underline{a}}}$ is not a $p$-th power.
\end{proof}

The next result explains how to find an $\bF_q$-point on a Frobenius nonclassical hypersurface $X$ by looking at its intersection with a suitable $\bF_q$-line. If $X$ contains all the $\bF_q$-lines of $\mathbb{P}^n$, then clearly $X$ contains all $\bF_q$-points in $\mathbb{P}^n$. Thus, it is natural to assume that there exists at least one $\bF_q$-line not contained in $X$.

\begin{lemma}
\label{lemma:incidence_non-pth-power}
Let $X=\{F=0\}\subset\bP^n$ be a Frobenius nonclassical hypersurface of degree $d$ over $\bF_q$, and let $p=\kchar(\bF_q)$. Assume that there exists an $\bF_q$-line $L\not\subset X$ such that the intersection $X\cap L$ is not a $p$-th power. Then $X\cap L$ contains at least one $\bF_q$-point and the intersection multiplicity at every non-$\bF_q$-point is divisible by $p$.
\end{lemma}

\begin{proof}
Assume without loss of generality that the line $L$ is given by 
\begin{align*}
    L = \{
        [x : y : 0 : 0 : \cdots : 0]
        \mid [x:y]\in \bP^1
    \}.
\end{align*}
Then $X\cap L$ is defined by the binary form $f(x,y)=F(x,y,0,...,0)$. Note that $f(x,y)\not\equiv 0$ because $L\not\subset X$. Because $X$ is Frobenius nonclassical,
$$
    F \mid x_0^q F_0 + ... + x_n^q F_n
    \qquad\text{where}\qquad
    F_i\colonequals\frac{\partial F}{\partial x_i},
$$
which implies that $f(x,y)$ divides $x^q f_x(x,y) + y^q f_y(x,y)$. By Euler's formula, $f(x,y)$ also divides $x f_x(x,y) + y f_y(x,y)$. As a result,
\begin{equation}
\label{eqn:f_div_fx-fy}
    f(x,y) \mid (x^q-x)f_x(x,y) + (y^q-y)f_y(x,y).
\end{equation}
Since $f(x,y)$ is not a $p$-th power by hypothesis, $f_x(x,y)$ and $f_y(x,y)$ cannot be identically zero simultaneously. Assume without loss of generality that $f_x(x,y)\neq 0$. Note that this implies $f_x(x,1)\not\equiv 0$. Substituting $y=1$ into \eqref{eqn:f_div_fx-fy}, we obtain
\begin{equation}
\label{eqn:f_div_f_x(x,1)}
    f(x,1) \mid (x^q-x)f_x(x,1).
\end{equation}

Suppose that $\alpha\in\overline{\bF_q}\setminus\bF_q$ is a non-$\bF_q$-root of $f(x,1)$, so that $f(x,1)=(x-\alpha)^m g(x)$ for some $m\geq 1$ and $g(\alpha)\neq 0$. Then
$$
    f_x(x,1)
    = m(x-\alpha)^{m-1} g(x)+(x-\alpha)^m g'(x)
    = (x-\alpha)^{m-1} h(x)
$$
where $h(x) = mg(x)+(x-\alpha) g'(x)$. Relation~\eqref{eqn:f_div_f_x(x,1)} now takes the form
$$
    (x-\alpha)^{m} g(x) \mid (x^q-x) (x-\alpha)^{m-1} h(x)
$$
which implies that $x-\alpha$ divides $(x^q-x) h(x)$.  The polynomial $(x-\alpha)$ does not divide $(x^q-x)$ since $\alpha\notin\bF_q$. Hence
$$
    (x-\alpha)\mid h(x) = mg(x)+(x-\alpha) g'(x).
$$
The displayed equation implies that $(x-\alpha)$ divides $mg(x)$. Since $g(\alpha)\neq 0$, we conclude that $m\equiv 0\pmod{p}$. Thus, every non-$\bF_q$-point in the intersection $X\cap L$ appears with multiplicity divisible by $p$.

Let us prove that there exists an $\bF_q$-point in the intersection $X\cap L$. If this is not true, then every $P\in X\cap L$ is not defined over $\bF_q$, whence appears with multiplicity divisible by $p$ due to the above result. But this implies that $f=F|_L$ is a $p$-th power, which contradicts the assumption that $X\cap L$ is not a $p$-th power. This completes the proof.
\end{proof}

\begin{cor}
\label{cor:meeting-with-line}
Let $X\subset\bP^n$ be a Frobenius nonclassical hypersurface over $\bF_q$ such that its degree $d\not\equiv 0 \pmod{p}$ where $p = \operatorname{char}(\bF_q)$. Then $X$ meets every $\bF_q$-line $L\subset\bP^n$ in at least one $\bF_q$-point.
\end{cor}

\begin{proof}
If $L\subset X$, then the proof is done. Assume $L\not\subset X$. The condition $d\not\equiv 0 \pmod{p}$ implies that $X\cap L$ is not a $p$-th power. The desired result follows from Lemma~\ref{lemma:incidence_non-pth-power}.
\end{proof}

\begin{cor}
\label{cor:rational-point}
Let $X = \{F=0\}\subset\bP^n$ be a Frobenius nonclassical hypersurface of degree~$d$ over $\bF_q$ and let $p = \operatorname{char}(\bF_q)$. If $d\leq q+1$ and $F$ is not a $p$-th power, then $X$ contains at least one $\bF_q$-point.
\end{cor}

\begin{proof}
Lemma~\ref{lemma:Bertini_not-pth-power} asserts the existence of an $\bF_q$-line $L\not\subset X$ such that $X\cap L$ is not a $p$-th power. The desired conclusion follows from Lemma~\ref{lemma:incidence_non-pth-power}.
\end{proof}

Corollary~\ref{cor:rational-point} fails in general when $d\geq q+2$. Indeed, the curves in Theorem~\ref{mainthm:upper-bound}~\ref{main:d=q+2} satisfy the hypothesis of the corollary except that $d = q+2$. Those curves contain no $\bF_q$-point by Corollary~\ref{cor:d=q+2_no-Fq-pt}.

\begin{rmk}
\label{rmk:pointless_example}
In Corollary~\ref{cor:rational-point}, the hypothesis that $F$ is not a $p$-th power is necessary because, if this assumption is dropped, then the $p$-th power of any hypersurface over $\bF_q$ that contains no $\bF_q$-point would provide a counterexample. One way to construct a concrete example of a pointless hypersurface $Y = \{G = 0\}\subset\bP^n$ is via the \emph{norm polynomials} \cite{LN96}*{Example~6.7} as follows. Let $\{\alpha_0, \alpha_1, ..., \alpha_{n}\}$ be a basis for $\bF_{q^{n+1}}$ over $\bF_q$. Consider the homogeneous polynomial over $\bF_q$:
$$
    G = \prod_{i=0}^{n}(
        \alpha_0^{q^i} x_0 + \cdots + \alpha_n^{q^i} x_n
    ).
$$
For every $\mathbf{b}=(b_0, b_1, ..., b_n)\in \bF_{q}^{n+1}$, it is easy to check that $G(\mathbf{b})=N_{\bF_{q^{n+1}}/\bF_{q}}(\mathbf{b})$ is the usual norm map. Hence $G(\mathbf{b})=0$ implies $\mathbf{b}=\mathbf{0}$. Thus, $Y=\{G=0\}\subset\bP^n$ contains no $\bF_q$-points. Another way to construct a pointless hypersurface when $n+1 < p$ is to take
$$
    G = \sum_{i=0}^{n}x_i^{q-1}.
$$
Because $a^{q-1}=1$ for all $a\in\bF_q^{\ast}$, it is clear that $Y=\{G=0\}\subset\bP^n$ has no $\bF_q$-points.
\end{rmk}

\subsection{Proof of the lower bound \texorpdfstring{$\mathbf{d\geq p+1}$}{d>=p+1}}
\label{subsect:d>=p+1}

Let us start by establishing two fundamental lemmas. The first one provides a criterion for the geometric integrality of a Frobenius nonclassical hypersurface.

\begin{lemma}
\label{lemma:sm-Fq-pt_geo-irr}
Let $X=\{F=0\}\subset\bP^n$ be a Frobenius nonclassical hypersurface over $\bF_q$ which is irreducible over $\bF_q$ and contains a smooth $\bF_q$-point $P\in X$. Then $X$ is geometrically integral.
\end{lemma}

\begin{proof}
Since $P$ is smooth, it is contained in a unique geometrically integral component $X'\subset X$. Under the $q$-th Frobenius endomorphism $\Phi$, we have
$
    P = \Phi(P)\in X'\cap\Phi(X').
$
Both $X'$ and $\Phi(X')$ are geometrically integral components of $X$ containing $P$, which ensures that $X' = \Phi(X')$. Thus $X'$ is defined over $\bF_q$. It follows that $X = X'$ as $X$ is irreducible over $\bF_q$.
\end{proof}

The next lemma gives a lower bound on the degree of a space-filling hypersurface.

\begin{lemma}
\label{lemma:space-filling_d-geq-q+1}
Let $X\subset\bP^n$ be a hypersurface over $\bF_q$. Suppose that $X$ is space-filling, namely, it satisfies $X(\bF_q)=\bP^n(\bF_q)$. Then $\deg(X)\geq q+1$.
\end{lemma}

\begin{proof}
We proceed by induction on $n$. When $n=1$, the conclusion follows since a space-filling subset $X\subset\bP^1$ is defined by a binary form divisible by $x^q y - xy^q$. For the inductive step, let $X\subset\bP^n$ be a hypersurface with $X(\bF_q)=\bP^n(\bF_q)$ where $n\geq 2$. If $X$ contains all the $\bF_q$-hyperplanes in $\bP^n$, then
$$
    \deg(X) \geq \#(\bP^n)^\ast(\bF_q) = \sum_{i=0}^nq^i > q+1.
$$
Otherwise, there exists an $\bF_q$-hyperplane $H\subset\bP^n$ such that $\dim(X\cap H) = \dim(H) - 1$. Now, $Y\colonequals X\cap H$ can be viewed as a hypersurface in $H\cong\bP^{n-1}$ which is space-filling. Applying the induction hypothesis to $Y$, we obtain $\deg(X) = \deg(Y) \geq q+1$.
\end{proof}

We are now ready to establish the lower bound $d\geq p+1$ in Theorem~\ref{mainthm:lower-bound}. The proof will proceed by induction on the dimension of the hypersurface. The following result settles the initial case.

\begin{lemma}
\label{lemma:lower-bound_p+1_curve}
Let $C\subset\bP^2$ be a Frobenius nonclassical curve of degree $d\geq 2$ over $\bF_q$ of characteristic $p$ which is smooth at all its $\bF_q$-points. Then $d\geq p+1$.
\end{lemma}

\begin{proof}
Because $C$ is smooth at $\bF_q$-points, it cannot contain more than one $\bF_q$-linear component, nor a nonreduced $\bF_q$-linear component. Since $\deg(C)\geq 2$, one can find $C'\subset C$ to be a curve irreducible over $\bF_q$ with degree $\geq 2$. Note that $C'$ is Frobenius nonclassical over $\bF_q$. There is nothing to prove if $\deg(C')\geq q+1$, so we can assume $\deg(C')\leq q$. If $C'$ is the $p$-th power of an $\bF_q$-line, then it is singular at $\bF_q$-points, a contradiction. Hence if $C'$ is a $p$-th power, then $\deg(C') = mp$ for some $m\geq 2$, thus $\deg(C')\geq p+1$ as desired. Assume that $C'$ is not a $p$-th power. Then Corollary~\ref{cor:rational-point} implies that $C'$ contains an $\bF_q$-point, which is smooth by hypothesis. It follows from Lemma~\ref{lemma:sm-Fq-pt_geo-irr} that $C'$ is geometrically integral.

The curve $C'$ is geometrically integral and thus is reduced. Therefore, $C'$ is nonreflexive by \cite{HV90}*{Proposition~1}. According to  \cite{HV90}*{Proposition 4} and \cite{SV86}*{Theorem 1.5}, there exists a non-$\mathbb{F}_q$-point $P\in C'$ such that the intersection multiplicity of $T_P C'$ with $C'$ at $P$ is a power of $p$. Since $C'$ is Frobenius nonclassical, $T_PC'$ also contains $\Phi(P)\neq P$ where $\Phi$ is the $q$-th Frobenius endomorphism. Therefore, $T_PC'$ intersects $C'$ in at least $p+1$ points counted with multiplicity, which yields $\deg(C)\geq\deg(C')\geq p+1$.
\end{proof}

\begin{thm}
\label{thm:lower-bound_p+1}
Let $X\subset\bP^n$, where $n\geq 2$, be a Frobenius nonclassical hypersurface of degree $d\geq 2$ over $\bF_q$ of characteristic $p$ which is smooth at all its $\bF_q$-points. Then $d\geq p+1$.
\end{thm}

\begin{proof}
Let us prove the statement by induction on $n$. The base case is done in Lemma~\ref{lemma:lower-bound_p+1_curve}. For the inductive step, note that since $X$ is smooth at $\bF_q$-points, it cannot contain more than one $\bF_q$-hyperplane, nor a nonreduced $\bF_q$-hyperplane. This fact, together with the hypothesis $d\geq 2$, guarantees the existence of an $\bF_q$-irreducible component $X'\subset X$ with degree $\geq 2$. By replacing $X$ with $X'$, we can further assume $X$ to be $\bF_q$-irreducible.

Take $H\subset \bP^n$ to be a hyperplane defined over $\bF_q$. Then $Y\colonequals X\cap H$ is a Frobenius nonclassical hypersurface in $H\cong\bP^{n-1}$. If $Y$ is smooth at all $\bF_q$-points, then the inductive hypothesis implies that $\deg(X) = \deg(Y)\geq p+1$.

Thus, we may assume that for every $H$ as above, the section $Y = X\cap H$ is singular at some $\mathbb{F}_q$-point $Q$, which is equivalent to asserting that $H = T_QY$. In other words, each $\bF_q$-hyperplane in $\mathbb{P}^n$ is tangent to $X$ at some $\mathbb{F}_q$-point. Notice that the Gauss map induced by $X$ is well-defined at the set of $\bF_q$-points as these points are smooth. It follows that the Gauss map is surjective at the level of $\mathbb{F}_q$-points:
$$\xymatrix{
    X(\bF_q) \ar@{->>}[r]
    & (\bP^n)^\ast(\bF_q) : P\ar@{|->}[r]
    & T_PX.
}$$
Hence $X$ is space-filling. By Lemma~\ref{lemma:space-filling_d-geq-q+1}, we get $\deg(X)\geq q+1\geq p+1$.
\end{proof}

\subsection{Blocking sets and abundance of rational points}
\label{subsect:many-rational-points}

In this part, we provide evidence for the abundance of $\mathbb{F}_q$-points on Frobenius nonclassical hypersurfaces. This phenomenon was already observed for the case of smooth plane curves by Hefez--Voloch~\cite{HV90} and later extended to singular curves by Borges--Homma~\cite{BH17}. We will start by showing that the configuration of $\bF_q$-points on a Frobenius nonclassical hypersurface possesses an interesting combinatorial structure. 

To state our first result in this direction, we introduce a relevant definition from finite geometry. 

\begin{defn}
Let $\mathcal{S}$ be a set of $\bF_q$-points in $\bP^n$. We say that $\mathcal{S}$ is a \emph{blocking set with respect to lines} if $\mathcal{S}\cap L$ is non-empty for each $\bF_q$-line $L\subset\bP^n$. Such a blocking set $\mathcal{S}$ is called \emph{trivial} if it contains all the $\bF_q$-points on some $\bF_q$-hyperplane. Otherwise, $\mathcal{S}$ is called \emph{non-trivial}.
\end{defn}

More generally, one can define the notion of a $k$-blocking set which is a set of $\bF_q$-points in $\bP^n$ which meets every $(n-k)$-dimensional space defined over $\bF_q$. The definition above can then be viewed as the special case when $k=n-1$. See \cite{KS20}*{Chapter 9} for a comprehensive account of finite geometry in higher dimensional spaces. For a further generalization and the sets of minimal cardinality under it, see \cite{Hub87}.  In addition, two recent papers \cites{AGY-blocking-1, AGY-blocking-2} study blocking sets arising from the $\mathbb{F}_q$-rational points of plane curves over finite fields.

\begin{prop}
\label{prop:blocking}
Let $X\subset\bP^n$ be a Frobenius nonclassical hypersurface over $\bF_q$ of degree $d\leq q$ and set $p = \operatorname{char}(\bF_q)$. If $d\not\equiv 0\pmod{p}$ and $X$ contains no $\bF_q$-linear component, then $X(\bF_q)$ is a non-trivial blocking set with respect to lines. Moreover, as long as $q\geq 5$, there is a lower bound
$$
    \#X(\bF_q) \geq \frac{q^{n}-1}{q-1}+\sqrt{q}\cdot q^{n-2}.
$$
\end{prop}

\begin{proof}
By Corollary~\ref{cor:meeting-with-line}, every $\bF_q$-line $L$ meets $X$ in at least one $\bF_q$-point, so $X(\bF_q)$ is a blocking set with respect to lines. Next, let $H$ be any $\bF_q$-hyperplane. Then $Y\colonequals X\cap H$ is a hypersurface of degree $d\leq q$ inside $H\cong\bP^{n-1}$. By Lemma~\ref{lemma:space-filling_d-geq-q+1}, we know that $Y(\bF_q)\neq H(\bF_q)$, which means $X(\bF_q)$ does not contain all of $H(\bF_q)$. This shows that $X(\bF_q)$ is a non-trivial blocking set. Finally, the lower bound on the number $\bF_q$-points follows from Huber's theorem on $(t,s)$-blocking set with $t=s=n-1$ \cite{Hub87}*{Theorem~1}. For additional references, see also \cite{Hei96} and \cite{HT15}*{Theorem~9.8}.
\end{proof}

Proposition~\ref{prop:blocking} fails when $d \geq q+1$ in view of Theorem~\ref{mainthm:upper-bound}. Indeed, for the examples in \ref{main:d=q+1_n-odd} and \ref{main:d=q+1_n-even} of the theorem, the former are space-filling, while the latter contain all the $\bF_q$-points on the hyperplane $\{x_0 = 0\}$. Therefore, the sets of $\bF_q$-points in these two cases form trivial blocking sets.

Next, we prove a lower bound on the number of $\bF_q$-points on a smooth Frobenius nonclassical surface that depends on the degree of the surface. As preparation, we prove a lower bound for curves.

\begin{lemma}
\label{lem:point-count-curve-smooth-at-rat-pts}
Let $C\subset\bP^2$ be a reduced Frobenius nonclassical curve over $\bF_q$ of degree $d\leq q+1$ which is smooth at all of its $\bF_q$-points. Then
$$
    \#C(\bF_q)\geq d(q-d+2).
$$
\end{lemma}

\begin{proof}
Let us write $C=C_1\cup \cdots\cup C_m$ where $C_i$, $i=1,\dots,m$ are $\bF_q$-irreducible. Since $C$ is reduced, each $C_i$ is not a $p$-th power. By Corollary~\ref{cor:rational-point} and Lemma~\ref{lemma:sm-Fq-pt_geo-irr}, each $C_i$ is geometrically integral. Denoting $d_i=\deg(C_i)$, we know that $\#C_i(\bF_q)\geq d_i(q-d_i+2)$ by \cite{BH17}*{Corollary~1.4}. Since $C$ is smooth at all of its $\bF_q$-points, $C(\bF_q)$ is a disjoint union of $C_i(\bF_q)$ for $i=1,\dots,m$. Therefore,
\begin{align*}
    \#C(\bF_q) 
    = \sum_{i=1}^{m} \#C_i(\bF_q) 
    \geq \sum_{i=1}^{m} d_i(q-d_i+2)
    &= dq - \sum_{i=1}^{m} d_i^2 + 2d \\
    &\geq dq - \left(\sum_{i=1}^{m} d_i\right)^2 + 2d
    = d(q-d+2),
\end{align*}
as claimed.
\end{proof}

We are now ready to establish a lower bound on the number of $\bF_q$-points for surfaces. We will see that, when $q$ is large, the bound is roughly
$$
    qd(q-d+2)\approx\textup{O}(dq^2).
$$
If the surface is not linear, then this bound grows at least at the rate $\textup{O}(q^{\frac{5}{2}})$ in view of the lower bound $d\geq\sqrt{q}+1$ in Theorem~\ref{mainthm:lower-bound}. This is a direct analogue of the curve case, where the lower bound is
$
    d(q-d+2)\approx\textup{O}(dq).
$

\begin{prop}
\label{prop:lower-bound-rat-points-surfaces}
Suppose that $X\subset\bP^3$ is a smooth Frobenius nonclassical surface over $\bF_q$ of degree $d \leq q+1$. Then
$$
    \#X(\bF_q)\geq
    \frac{(q^3+q^2+q+1)\cdot d(q-d+2)}{q^2+q+d(q-d+2)}. 
$$
\end{prop}

\begin{proof}
Let us call an $\bF_q$-plane $H$ ``good'' if the hyperplane section $X\cap H$ is smooth at all $\bF_q$-points. Consider the set
$$
    \mathcal{I} = \left\{
        (H, P) \mid 
        H\text{ is a good plane with } P\in (X\cap H)(\bF_q)
    \right\}.
$$
The number of good planes is at least
$$
    \#(\bP^3)^\ast(\bF_q) - \#X(\bF_q)
    = (q^3+q^2+q+1)-\#X(\bF_q).
$$
For each good plane $H$, we observe that
$
    \# (X\cap H)(\bF_q)\geq d(q-d+2)
$
by Lemma~\ref{lem:point-count-curve-smooth-at-rat-pts}. Here, we are using the fact that a plane section of a smooth surface is a reduced curve (which is a consequence of Zak's theorem \cite{Zak93}*{Corollary~I.2.8}). Thus, we get a lower bound
$$
    \#\mathcal{I} \geq (q^3+q^2+q+1-\#X(\bF_q)) \cdot d(q-d+2).
$$
On the other hand, each $\bF_q$-point of $X$ is contained in at most $q^2+q$ good planes. This is because there are $q^2+q+1$ planes over $\bF_q$ passing through such a point and one of them is the tangent plane. This gives us an upper bound,
$$
    \#\mathcal{I} \leq \#X(\bF_q) \cdot (q^2+q).
$$
Combining the lower and the upper bounds, we obtain,
$$
    (q^3+q^2+q+1-\#X(\bF_q)) \cdot d(q-d+2)
    \leq \#X(\bF_q) \cdot (q^2+q).
$$
Rearranging this inequality, we get the desired result. 
\end{proof}

We end this section with a remark about upper bounds on the number of rational points.

\begin{rmk}
In \cite{HK13_bound}*{Theorem~1.2}, Homma and Kim proved an upper bound for the number of rational points on a hypersurface without an $\bF_q$-linear component. According to \cite{Tir17}*{Theorem~1~(1) and (2)}, this upper bound is achieved by examples in Theorem~\ref{mainthm:upper-bound}~\ref{main:d=q+1_n-odd} and surface examples of degree $\sqrt{q}+1$ under the setting of Theorem~\ref{mainthm:lower-bound}.
\end{rmk}

\section{Lower bounds on degree and Hermitian surfaces}
\label{sect:degree_lower_bound}

In this section, we finish the proof of Theorem~\ref{mainthm:lower-bound}. For the lower bound $d\geq\sqrt{q}+1$, we will present a proof by contradiction, and therefore, will assume the existence of a hypersurface $X\subset\bP^n$ as in the hypothesis except that $d\leq\sqrt{q}$. Our strategy consists of two steps:
\begin{enumerate}[label=\Roman*.]
    \item\label{step1:sharp-lower-bound}
    Finding a $2$-plane $H\subset\bP^n$ over $\bF_q$ such that $X\cap H$ contains a curve component $C$ over $\bF_q$ that is reduced of degree $\geq 2$ and smooth at all its $\bF_q$-points.
    \item\label{step2:sharp-lower-bound}
    Finding a curve $C'\subset C$ over $\bF_q$ of degree $\geq 2$ which is geometrically integral. This curve is Frobenius nonclassical by construction. Hence $\deg(C')\geq\sqrt{q}+1$ by \cite{BH17}*{Corollary~3.2}, contradicting our assumption that $d\leq\sqrt{q}$.
\end{enumerate}
For the last statement in the theorem, the curve case is done in \cite{BH17}*{Corollary~3.2}, so we will prove the assertion in the surface case. In the end of this section, we include a discussion on the Hermiticity in higher dimensions.

\subsection{Linear sections that are smooth at \texorpdfstring{$\bF_q$}{Fq}-points}
\label{subsect:smooth-Fq-points}

Here we prove several Bertini-type results for reduced hypersurfaces over $\bF_q$ that are smooth at all $\bF_q$-points, which will be used to establish Step~\ref{step1:sharp-lower-bound} in our strategy.

\begin{lemma}
\label{lemma:weak-bertini-n>=4}
Let $X\subset\bP^n$ where $n\geq 4$ be a reduced hypersurface over $\bF_q$ of degree $d\leq\frac{q}{2}$ on which every $\bF_q$-point is smooth. Then there exists an $\bF_q$-hyperplane $H$ such that $X\cap H$ is reduced of dimension $n-2$ and smooth at all $\bF_q$-points.
\end{lemma}

\begin{proof}
Our strategy is to show that the number of $\bF_q$-hyperplanes in $\bP^n$ is greater than the number of $\bF_q$-hyperplanes $H$ that satisfy at least one of the following bad conditions:
\begin{itemize}
    \item $X\cap H$ is not reduced.
    \item $X\cap H$ is not of dimension $n-2$. 
    \item $X\cap H$ is singular at some $\bF_q$-point. 
\end{itemize}
By \cite{ADL24hyper}*{Proposition~4.6}, the number of $\bF_q$-hyperplanes $H$ such that $X\cap H$ is not reduced or does not have dimension $n-2$ is at most
$$
    d(d-1)(q+1)^2 + 1.
$$
On the other hand, $X\cap H$ is singular at an $\bF_q$-point $P\in X$ implies that $H = T_PX$, so the number of hyperplanes $H$ for which $X\cap H$ is singular at some $\bF_q$-point is at most $\#X(\bF_q)$, the total number of $\bF_q$-points on $X$. It was proved independently by Serre \cite{Ser91} and Sorensen \cite{Sor94} that
$$
    \#X(\bF_q)\leq dq^{n-1} + q^{n-2} + \cdots + q + 1.
$$
Thus, the number of bad hyperplanes is at most the sum of these contributions:
\begin{align*}
     \left(
        dq^{n-1} + q^{n-2} + \cdots + q + 1
    \right)
    + d(d-1)(q+1)^2 + 1.
\end{align*}

Because the total number of $\bF_q$-hyperplanes in $\bP^n$ is $\sum_{i=0}^nq^{i}$, the result follows if we can prove the following inequality:
$$
    \sum_{i=0}^nq^{i}
    > \left(
        dq^{n-1} + q^{n-2} + \cdots + q + 1
    \right) + d(d-1)(q+1)^2 + 1
$$
which is equivalent to
\begin{align}
\label{eq:good-planes-exist}
    q^{n} > (d-1) q^{n-1} +  d(d-1)(q+1)^2 + 1.
\end{align}
Using the assumptions that $d\leq\frac{q}{2}$ and $n\geq 4$, we get
\begin{align*}
    (d-1)q^{n-1} + d(d-1)(q+1)^2 + 1
    & \leq \left(\frac{q}{2} - 1\right)q^{n-1}
    + \frac{q}{2}\left(\frac{q}{2} - 1\right)(q+1)^2
    + 1 \\
    & = \frac{q^n}{2} - q^{n-1}
    + \frac{q^4}{4} - \frac{3q^2}{4} - \frac{q}{2}
    + 1
    < q^n.
\end{align*}
The last inequality can be proved by computing directly that the real function
$$
    f(x) = x^n - \left(
        \frac{x^n}{2} - x^{n-1} + \frac{x^4}{4} - \frac{3x^2}{4} - \frac{x}{2} + 1
    \right)
    \quad\text{where}\quad
    n\geq 4
$$
satisfies $f(x) > 0$ for all $x\geq 2$. This proves inequality~\eqref{eq:good-planes-exist} and establishes the existence of an $\bF_q$-hyperplane satisfying the desired conditions.
\end{proof}

\begin{lemma}
\label{lemma:weak-bertini-n=3}
Let $X\subset \bP^3$ be a reduced surface of degree $d$ over $\bF_q$ such that $2\leq d\leq\sqrt{q}$ and every $\bF_q$-point on $X$ is smooth. Then there exists an $\bF_q$-plane $H$ such that $X\cap H$ contains a curve component $C$ over $\bF_q$ of degree $\geq 2$ that is reduced and smooth at its $\bF_q$-points.
\end{lemma}

\begin{proof}
The proof is similar to that of Lemma~\ref{lemma:weak-bertini-n>=4}. The only difference is that we need to apply a refined bound on the number of $\bF_q$-points on $X$ by Homma and Kim \cite{HK13_bound}. The theorem \cite{HK13_bound}*{Theorem~1.2} in the case of surfaces asserts that
\begin{align}
\label{eq:hk-bound-surfaces}
    \# X(\bF_q) \leq (d-1)q^{2}+d q + 1
\end{align}
provided that $X$ has no $\bF_q$-linear component. Since $X$ has degree $\geq 2$ and is smooth at all $\bF_q$-points, it contains at most one $\bF_q$-plane counted with multiplicity. This implies that $X$ contains an $\bF_q$-irreducible components $X'$ satisfying the assumptions of the lemma. Hence, by replacing $X$ with $X'$ if necessary, we can assume that $X$ is irreducible over $\bF_q$. This in turn allows us to apply~\eqref{eq:hk-bound-surfaces} directly.

As in the proof of Lemma~\ref{lemma:weak-bertini-n>=4}, the number of bad planes defined over $\bF_q$ is at most
\begin{align*}
    \# X(\bF_q) +  d(d-1)(q+1)^2 + 1
    \leq (d-1)q^{2}+d q + 1 +  d(d-1)(q+1)^2 + 1.
\end{align*}
The conclusion will follow if we can show that the total number of $\bF_q$-planes exceeds the number of bad planes. Thus, it suffices to show that:
\begin{align}\label{eq:hk-bound-surface-step1}
    q^{3}+q^{2}+q+1 > (d-1)q^{2}+d q + 1 +  d(d-1)(q+1)^2 + 1
\end{align}
whenever $d\leq \sqrt{q}$. The right hand side of the inequality increases as $d$ increases, so it suffices to establish the inequality \eqref{eq:hk-bound-surface-step1} in the case $d=\sqrt{q}$. In other words, it suffices to prove the inequality
$$
     q^{3} + q^{2} + q + 1
     > (\sqrt{q}-1)q^{2} + \sqrt{q}\cdot q + 1
     + \sqrt{q}(\sqrt{q}-1)(q+1)^2 + 1.
$$
This inequality can be directly checked for each $q\geq 2$. This justifies \eqref{eq:hk-bound-surface-step1} and furnishes an $\bF_q$-plane $H$ such that $C = H\cap X$ is a reduced curve of degree $d\geq 2$ that is smooth at all its $\bF_q$-points.
\end{proof}

\begin{cor}
\label{cor:curve-section-sm-Fq}
Let $X\subset\bP^n$, where $n\geq 3$, be a reduced hypersurface over $\bF_q$ of degree $d$ such that $2\leq d\leq\sqrt{q}$ and every $\bF_q$-point on $X$ is smooth. Then there exists an $\bF_q$-plane $H$ such that $X\cap H$ contains a curve component $C$ over $\bF_q$ of degree $\geq 2$ that is reduced and smooth at all its $\bF_q$-points.
\end{cor}

\begin{proof}
The assumption $2\leq\sqrt{q}$ is equivalent to $\sqrt{q}\leq\frac{q}{2}$, so we can apply Lemma~\ref{lemma:weak-bertini-n>=4} repeatedly until getting a linear subspace $H'\cong\bP^3$ over $\bF_q$ such that $X\cap H'$ is a reduced surface that is smooth at all $\bF_q$-points. By Lemma~\ref{lemma:weak-bertini-n=3}, there exists a plane $H\subset H'$ over $\bF_q$ such that $X\cap H$ satisfies the desired property.
\end{proof}

\subsection{Existence of transverse lines to plane curves}
\label{subsect:tran-line_curve}

As an intermediate step, we establish a few results that guarantee the existence of a transverse $\bF_q$-line to a reduced plane curve with smooth $\bF_q$-points under the assumption that $d\leq\sqrt{q}$.

\begin{lemma}
\label{lemma:count-of-non-transverse-lines-geom-irred}
Let $C\subset\bP^2$ be a geometrically integral curve of degree $d$ over $\overline{\mathbb{F}_q}$ that is smooth at all $\bF_q$-points. Then the number of $\bF_q$-lines that are not transverse to $C$ is at most
$$
    \frac{1}{2}(d-1)(d-2) + d(d-1)q + 1.
$$
\end{lemma}

\begin{proof}
A line $L$ is not transverse to the curve $C$ if and only if it passes through a singular point of $C$ or is the tangent line at a smooth point of $C$. We will compute an upper bound for the number of $\bF_q$-lines in each of the two categories. 

Recall that the number of singular points of a geometrically integral plane curve is bounded by its arithmetic genus $\frac{1}{2} (d-1)(d-2)$.
By hypothesis, each singular point of $C$ is \emph{not} defined over $\bF_q$. Thus, each singular point has at most one $\bF_q$-line passing through it, which gives a total contribution of at most
\begin{equation}
\label{eqn:singuler-point-on-curve}
    \frac{1}{2}(d-1)(d-2)
\end{equation}
many $\bF_q$-lines passing through a singular point of $C$.

To estimate the number of tangent $\bF_q$-lines to $C$, note that the dual curve $C^{\ast}$ has degree at most $d(d-1)$. It follows that
\begin{equation}
\label{eqn:tangent-to-curve}
    \#\{\bF_q\text{-lines tangent to }C\}
    \leq \#C^{\ast}(\bF_q)
    \leq \deg(C^{\ast})q + 1
    \leq d(d-1)q + 1
\end{equation}
where the second inequality follows by applying the Serre--S\o rensen bound to $C^{\ast}$. (The bound says that $\#E(\bF_q)\leq\delta q+1$ for any plane curve $E$ of degree $\delta$.) Adding up \eqref{eqn:singuler-point-on-curve} and \eqref{eqn:tangent-to-curve} gives the desired bound.
\end{proof}

\begin{rmk}
In the proof above, we used the result $\# E(\bF_q)\leq dq+1$ for any plane curve $E\subset\bP^2$. This result is a special case of the more general result for hypersurfaces independently proved by Serre~\cite{Ser91} and S\o rensen~\cite{Sor94}. However, in those papers, the varieties are defined over $\bF_q$, whereas in our lemma, $E$ is a plane curve defined over $\overline{\bF_q}$, and not necessarily over $\bF_q$. The result $\# E(\bF_q)\leq dq+1$ nonetheless holds even in this more general case. Indeed, the same proof in S\o rensen~\cite{Sor94}*{Theorem 2.1} goes through even if $E$ is not defined over $\bF_q$.
\end{rmk}

Next, we generalize the previous result without the hypothesis on geometric irreducibility.

\begin{lemma}
\label{lemma:count-of-non-transverse-lines}
Let $C\subset \bP^2$ be a reduced curve of degree $d$ over $\overline{\mathbb{F}_q}$ on which every $\bF_q$-point is smooth. Then the number of $\bF_q$-lines not transverse to $C$ is at most
$$
    \frac{1}{2}d^2 + qd^2 - qd + \frac{1}{2}d.
$$
\end{lemma}

\begin{proof}
Let us write $C = C_1\cup C_2\cup\cdots\cup C_m$ where each $C_i$ is geometrically integral of degree~$d_i$. Every $\bF_q$-line not transverse to $C$ is either not transverse to some $C_i$ (Type~I) or passes through an intersection of $C_i$ and $C_j$ for some $i\neq j$ (Type~II). The number of non-transverse $\bF_q$-lines of Type I can be bounded by applying Lemma~\ref{lemma:count-of-non-transverse-lines-geom-irred} to each $C_i$ and summing up the contributions:
$$
    \#\{\text{Type~I non-transverse } \bF_q\text{-lines}\}
    \leq\sum_{i=1}^{m}\left(
        \frac{1}{2}(d_i-1)(d_i-2)+d_i(d_i-1)q+1
    \right).
$$
To give an upper bound on the number of non-transverse $\bF_q$-lines of Type II, we note that any point $P\in C_i\cap C_j$ is a singular point of $C$. Since $C$ is smooth at all $\bF_q$-points, $P$ is not defined over $\bF_q$. Thus, there can be at most one $\bF_q$-line that passes through $P$. Since $C_i\cap C_j$ has at most $d_id_j$ distinct points by B\'ezout's theorem,
$$
    \#\{\text{Type II non-transverse } \bF_q\text{-lines}\}
    \leq\sum_{i<j}d_id_j.
$$
Therefore, the number of $\bF_q$-lines not transverse to $C$ is at most
\begin{align*}
    & \sum_{i=1}^{m} \left(\frac{1}{2}(d_i-1)(d_i-2)+d_i(d_i-1)q+1\right) + \sum_{i<j}d_i d_j \\
    &= \sum_{i=1}^{m} \left(\frac{1}{2}(d_i^2-3d_i+2)+(d_i^2-d_i)q+1\right) + \sum_{i<j}d_i d_j \\ 
    &=\frac{1}{2}\left(
        \sum_{i=1}^{m} d_i^2+\sum_{i<j} 2d_id_j
    \right)
    + \sum_{i=1}^{m} q d_i^2
    - \sum_{i=1}^{m} \left(q + \frac{3}{2}\right)d_i
    + \sum_{i=1}^{m} 2 \\
    &\leq \frac{1}{2}d^2 + qd^2 - \left(q + \frac{3}{2}\right)d+2m \\
    &= \frac{1}{2}d^2+qd^2 - qd + \frac{1}{2}d + (2m-2d) \\
    &\leq\frac{1}{2}d^2+qd^2-qd + \frac{1}{2}d
\end{align*}
where the last inequality follows from the fact that $m\leq d$.
\end{proof}

\begin{cor}
\label{cor:transverse-line-plane-curves}
Let $C\subset \bP^2$ be a reduced curve over $\overline{\mathbb{F}_q}$ of degree $d\leq\sqrt{q}$ that is smooth at all $\bF_q$-points. Then there exists an $\bF_q$-line $L$ transverse to $C$.
\end{cor}
\begin{proof}
It suffices to show that the total number of $\bF_q$-lines in $\bP^2$ exceeds the number of $\bF_q$-lines not transverse to $C$. By Lemma~\ref{lemma:count-of-non-transverse-lines} and the hypothesis that $d\leq\sqrt{q}$,
\begin{align*}
    &\#\{\bF_q\text{-lines not transverse to }C\}
    \leq\frac{1}{2}d^2+qd^2 - \left(q - \frac{1}{2}\right)d
    < \frac{1}{2}d^2+qd^2 \\
    &\leq \frac{1}{2}\left(\sqrt{q}\right)^2 +q\left(\sqrt{q}\right)^2
    = q^2+\frac{q}{2}
    < q^2 + q + 1 = \#\{\bF_q\text{-lines in } \bP^2\},
\end{align*}
as desired.
\end{proof}

\begin{rmk}
In \cite{ADL24hyper}*{Theorem~4.2}, we proved that every reduced plane curve $C$ of degree $d$ in $\bP^2$ admits a transverse $\bF_q$-line once $q\geq \frac{3}{2}d(d-1)$. Corollary~\ref{cor:transverse-line-plane-curves} relaxes the bound to $q\geq d^2$ at the cost of the additional hypothesis that $C$ is smooth at its $\bF_q$-points.
\end{rmk}

\subsection{Proof of the lower bound \texorpdfstring{$\mathbf{d\geq\sqrt{q}+1}$}{d>=sqrt(q)+1}}
\label{subsect:geo-irr_lower-bound}

To establish Step~\ref{step2:sharp-lower-bound}, we need one more result concerning the existence of a smooth $\bF_q$-point and geometric integrality for Frobenius nonclassical curves.

\begin{lemma}
\label{lemma:tran-line_Fq-only}
Let $C\subset \bP^2$ be a Frobenius nonclassical curve over $\bF_q$ which is irreducible over $\bF_q$ and admits a transverse $\bF_q$-line $L$. Then $L$ intersects $C$ in $\bF_q$-points only and $C$ is geometrically integral.
\end{lemma}

\begin{proof}
Assume, to the contrary, that there exists $P\in L\cap C$ \emph{not} defined over $\bF_q$. Note that $P$ is a smooth point as $L$ meets $C$ transversely. Under the $q$-th Frobenius endomorphism $\Phi$, we have $\Phi(P)\in T_PC\cap L$. As a result, both $T_PC$ and $L$ contain the distinct points $P$ and $\Phi(P)$, whence $L = T_PC$. This shows that $L$ is tangent to $C$, a contradiction. Therefore, $L$ intersects $C$ in $\deg(C)$ many smooth $\bF_q$-points, which implies that $C$ is geometrically integral by Lemma~\ref{lemma:sm-Fq-pt_geo-irr}.
\end{proof}

We are now ready to prove the lower bound $d\geq\sqrt{q}+1$ in Theorem~\ref{mainthm:lower-bound}.

\begin{thm}
\label{thm:lower-bound_sqrt(q)+1}
Let $X\subset\bP^n$, where $n\geq 2$, be a reduced Frobenius nonclassical hypersurface over $\bF_q$ of degree $d\geq 2$ which is smooth at all its $\bF_q$-points. Then $d\geq\sqrt{q}+1$.
\end{thm}

\begin{proof}
Assume, to the contrary, that there exists a reduced Frobenius nonclassical hypersurface $X\subset\bP^n$ of degree $d$ where $2\leq d\leq\sqrt{q}$. By Corollary~\ref{cor:curve-section-sm-Fq}, there exists an $\bF_q$-plane $H$ such that $X\cap H$ contains a curve component $C$ over $\bF_q$ that is reduced of degree $\geq 2$ and smooth at all $\bF_q$-points. The smoothness at $\bF_q$-points forbids $C$ from being a union of $\bF_q$-lines, so there exists a curve $C'\subset C$ of degree $\geq 2$ irreducible over $\bF_q$.

Being an $\bF_q$-component of a linear section over $\bF_q$ implies that $C'$ is Frobenius nonclassical as a plane curve in $H\cong\bP^2$. On the other hand, $C'$ is smooth at all $\bF_q$-points and we have $\deg(C')\leq\deg(C)\leq\deg(X)\leq\sqrt{q}$, so there exists an $\bF_q$-line transverse to $C'$ by Corollary~\ref{cor:transverse-line-plane-curves}. We conclude that $C'$ is geometrically integral by Lemma~\ref{lemma:tran-line_Fq-only}. However, this implies $\deg(C')\geq\sqrt{q}+1$ according to \cite{BH17}*{Corollary~3.2}, a contradiction.
\end{proof}

\subsection{Frobenius nonclassical surfaces of minimal degrees}
\label{subsect:FrobNoncalSurf_Hermition}

It turns out that smooth Frobenius nonclassical surfaces over $\bF_q$ whose degrees attain the minimum $\sqrt{q}+1$ are Hermitian, that is, they are projectively equivalent over $\bF_q$ to the surface defined by
$$
    x_0^{\sqrt{q}+1} + x_1^{\sqrt{q}+1} + x_2^{\sqrt{q}+1} + x_3^{\sqrt{q}+1} = 0.
$$
Our proof is built upon the curve case proved by Borges and Homma \cite{BH17}*{Corollary~3.2}, which in turn relies on the characterization of Hermitian curves \cite{HKT08}*{Theorem~10.47}. The main idea of our proof in the surface case is to find sufficiently many $\bF_q$-planes that cut out Hermitian curves on the surface.

\begin{lemma}
\label{lemma:min-degree-curve}
Let $C\subset\bP^2$ be a reduced Frobenius nonclassical curve over $\bF_q$ of degree $d=\sqrt{q}+1$ which is smooth at all its $\bF_q$-points. Then $C$ is Hermitian.
\end{lemma}

\begin{proof}
Because $C$ is smooth at its $\bF_q$-points, it contains at most one $\bF_q$-line component counted with multiplicity. As $\deg(C)\geq 3$, we can find an $\bF_q$-irreducible component $C'\subset C$ with $\deg(C')\geq 2$. Since $C'$ has degree at most $\sqrt{q}+1$ and is not a $p$-th power (as $C$ is reduced), it contains an $\bF_q$-point $P'$ by Corollary~\ref{cor:rational-point}. By hypothesis, $C'$ is smooth at $P'$, so $C'$ is geometrically integral by Lemma~\ref{lemma:sm-Fq-pt_geo-irr}. Applying \cite{BH17}*{Corollary~3.2}, we conclude that $\deg(C')\geq\sqrt{q}+1=\deg(C)$. Thus $C = C'$, and the same source shows that $C = C'$ is Hermitian.
\end{proof}

\begin{lemma}
\label{lemma:Fq-sm-section}
Let $X\subset\bP^n$ be a hypersurface over $\bF_q$ of degree $d$ which is smooth at all its $\bF_q$-points. Then the number of $\bF_q$-hyperplanes $H\subset\bP^n$ such that the intersection $X\cap H$ is proper and smooth at all $\bF_q$-points is at least
$
    q^{n-1}(q-d+1).
$
\end{lemma}

\begin{proof}
The number of $\bF_q$-hyperplanes that are tangent to $X$ at an $\bF_q$-point or appear as a linear component of $X$ is bounded by $\#X(\bF_q)$. Using the Serre--S\o rensen bound \cites{Ser91,Sor94},
$$
    \# X(\bF_q)\leq 
    dq^{n-1} + \frac{q^{n-1}-1}{q-1}.
$$
Therefore, the number of $\bF_q$-hyperplanes as in the statement is at least
$$
    \frac{q^{n+1}-1}{q-1} - \left(
        dq^{n-1} + \frac{q^{n-1}-1}{q-1}
    \right)
    = \left(\frac{q^2-1}{q-1}\right)q^{n-1} - dq^{n-1}
    = q^{n-1}(q-d+1)
$$
which gives the desired bound.
\end{proof}

\begin{prop}
\label{prop:Hermitian-surf}
Suppose that $X\subset\bP^3$ is a smooth Frobenius nonclassical surface defined over $\bF_q$ of degree $d=\sqrt{q}+1$. Then $X$ is Hermitian.
\end{prop}

\begin{proof}
Note that $X$ contains no $\bF_q$-linear component due to the hypothesis. Let us call an $\bF_q$-plane $H\subset\bP^3$ ``good'' if $X\cap H$ is smooth at all $\bF_q$-points. By Lemma~\ref{lemma:Fq-sm-section}, $X$ admits at least $q^3-q^2\sqrt{q}$ many good planes. Let $H\subset\bP^3$ be any good plane. Then $C \colonequals X\cap H$ is a Frobenius nonclassical curve over $\bF_q$ of degree $\sqrt{q}+1$ which is smooth at all $\bF_q$-points. As $X$ is smooth, $C$ is a reduced curve by Zak's theorem \cite{Zak93}*{Corollary~I.2.8}. Thus $C$ is Hermitian by Lemma~\ref{lemma:min-degree-curve}.

Note that $q\geq 4$ since $q$ is a square. We claim that there exist four good planes $H_0$, $H_1$, $H_2$, $H_3$ that are linearly independent. Indeed, if this was not true, then the number of good planes would be at most $\#(\bP^2)^\ast(\bF_q) = q^2+q+1$, which is strictly less than $q^3-q^2\sqrt{q}$ for $q\geq 4$, a contradiction. After a change of coordinates, we may assume $H_i = \{x_i = 0\}$ for $i=0,1,2,3$. Let $F$ be a defining polynomial for $X$. Then the fact that $X\cap\{x_i=0\}$ is Hermitian implies that
\begin{itemize}
    \item For each $i$, $F|_{x_i=0}$ defines a Hermitian curve in the remaining three variables.
    \item $F$ does not contain any monomials involving exactly three variables.
\end{itemize}
By collecting the monomials appropriately, we can express
\begin{equation}
\label{hermitian-surface-decomposition}
    F = G + x_0x_1x_2x_3 R
\end{equation}
where $G$ defines a Hermitian surface and $\deg(R)=\sqrt{q}-3$.

Since $X$ admits at least $q^3-q^2\sqrt{q}$ good planes and we have used up $4$ of those planes in the analysis above, there are still at least $q^3-q^2\sqrt{q}-4$ good planes remaining. If
$$
    W = \{a_0x_0+a_1x_1+a_2x_2+a_3x_3 = 0\}
$$
defines a good plane other than $H_0,\dots,H_3$, then, by imposing the relation $\sum_ia_ix_i = 0$ on equation~\eqref{hermitian-surface-decomposition}, we deduce that
$$
    R|_W = 0.
$$
Indeed, if we write, without loss of generality, that $x_0 = -a_0^{-1}(a_1x_1+a_2x_2+a_3x_3)$, then the term $x_0x_1x_2x_3 R$ contributes a nonzero monomial to $F|_W$ which involves $x_1^{i} x_2^{j} x_3^{k}$ with $i,j,k>0$. But such a term cannot appear in a polynomial defining a Hermitian curve.

Therefore, $R$ is divisible by $\sum_ia_ix_i$. Since $W$ is an arbitrary good plane different from $H_0,\dots,H_3$, we conclude that $R$ is divisible by a product of distinct $q^3-q^2\sqrt{q}-4$ linear forms over $\bF_q$. Suppose, for contradiction, that $R\neq 0$. Then we have
$$
    \sqrt{q}-3 = \deg(R)
    \geq q^3-q^2\sqrt{q}-4.
$$
However, $\sqrt{q}-3 < q^3-q^2\sqrt{q}-4$ for $q\geq 4$. This leads to a contradiction, thus $R=0$. We conclude that $F = G$, and thus $X$ is a Hermitian surface.
\end{proof}

\subsection{Evidence for the Hermiticity in higher dimensions}
\label{subsect:FrobForm_Hermitian}

For most examples of Frobenius nonclassical hypersurfaces that we know, their defining polynomials $F$ satisfy the property that $F_{1,0}$ is proportional to a power of $F$. These include all examples in this paper except for the ones classified by Theorem~\ref{mainthm:upper-bound}~\ref{main:d=q+1_n-even}, and also the curve examples in \cite{HV90}*{Theorem~2}. Therefore, it is reasonable to assume this condition while investigating the properties of Frobenius nonclassical hypersurfaces. In the following, we provide evidence for Conjecture~\ref{conj:hermitian} by showing that it holds in odd characteristics under this assumption. The proof relies on the main results from \cite{KKPSSVW22} about $F$-pure thresholds.

\begin{prop}
\label{prop:F10=powerF}
Let $F\in\bF_q[x_0,\dots,x_n]$ be a homogeneous polynomial of degree $\sqrt{q}+1$ which defines a reduced Frobenius nonclassical hypersurface that satisfies
$$
    F_{1,0}
    \colonequals\sum_{i=0}^nx_i^q\frac{\partial F}{\partial x_i}
    = cF^{\sqrt{q}}
    \qquad\text{for some}\qquad
    c\in\bF_q\setminus\{0\}.
$$
Then $F$ is a Frobenius form, namely, it is defined by the expression
$$
    F = \sum_{i=0}^nx_i^{\sqrt{q}}L_i
$$
for some linear polynomials $L_0,\dots,L_n$.
\end{prop}

\begin{proof}
Let $\mathrm{fpt}(F)$ and $\mathrm{fpt}(F_{1,0})$ denote the $F$-pure thresholds of $F$ and $F_{1,0}$, respectively. First, we have $\mathrm{fpt}(F)\geq\frac{1}{\sqrt{q}}$ by \cite{KKPSSVW22}*{Theorem~1.1}. Next, the assumption $F_{1,0} = cF^{\sqrt{q}}$ implies that $\mathrm{fpt}(F_{1,0}) = \frac{1}{\sqrt{q}}\mathrm{fpt}(F)$ by \cite{KKPSSVW22}*{Proposition~2.2~(2)}. Moreover, we see that $\mathrm{fpt}(F_{1,0})\leq\frac{1}{q}$ from the definition of $F_{1,0}$. Combining these relations, we obtain
$$
    \frac{1}{\sqrt{q}}
    \leq\mathrm{fpt}(F)
    = \sqrt{q}\cdot\mathrm{fpt}(F_{1,0})
    \leq\frac{1}{\sqrt{q}}
$$
which implies that $\mathrm{fpt}(F) = \frac{1}{\sqrt{q}}$. Thus $F$ is a Frobenius form by \cite{KKPSSVW22}*{Theorem~4.3}.
\end{proof}

Now we proceed to prove that the Frobenius form in Proposition~\ref{prop:F10=powerF} turns out to be Hermitian due to the Frobenius nonclassical property. First note that a Frobenius form over an arbitrary field $k$ of positive characteristic $p$ can be equivalently written as
\begin{equation}
\label{eqn:frobForm}
    F = \sum_{i,j=0}^nx_i^{q'}M_{ij}x_j
    = \mathbf{x}^{q'}\cdot M\cdot\mathbf{x}^t
\end{equation}
where $\mathbf{x} = (x_0,\dots,x_n)$, $\mathbf{x}^{q'} = (x_0^{q'},\dots,x_n^{q'})$, and $M = (M_{ij})$ is a matrix with entries in $k$.

\begin{lemma}
\label{lemma:frobForm_Herm_skewHerm}
Let $F$ be a Frobenius form as in \eqref{eqn:frobForm} over $\bF_{{q'}^2}$. Suppose that the characteristic $p\neq 2$ and that $F$ divides the polynomial
$$
    G\colonequals
    \sum_{i=0}^nx_i^{{q'}^2}\frac{\partial F}{\partial x_i}
    = \mathbf{x}^{q'}\cdot M\cdot (\mathbf{x}^{{q'}^2})^t.
$$
Then, upon rescaling by a constant $c\in\bF_{{q'}^2}$, the matrix $M$ is either
\begin{itemize}
    \item a Hermitian matrix in the sense that $\overline{M}\colonequals (M_{ij}^{q'}) = M^t$, or
    \item a skew-Hermitian matrix in the sense that $\overline{M}\colonequals (M_{ij}^{q'}) = -M^t$.
\end{itemize}
\end{lemma}

\begin{proof}
Because $p\neq 2$, we are allowed to write $M = M_1 + M_2$ where
$$
    M_1 = \frac{1}{2}(M + \overline{M}^t)
    \qquad\text{and}\qquad
    M_2 = \frac{1}{2}(M - \overline{M}^t).
$$
Notice that $M_1$ is Hermitian and $M_2$ is skew-Hermitian. Now we have $F = F_1 + F_2$ where $F_1 = \mathbf{x}^{q'}\cdot M_1\cdot\mathbf{x}^t$ and $F_2 = \mathbf{x}^{q'}\cdot M_2\cdot\mathbf{x}^t$. One can verify directly that $G = F_1^{q'} - F_2^{q'}$. This relation, together with the fact that $F^{q'} = F_1^{q'} + F_2^{q'}$, implies
$$
    G + F^{q'} = 2F_1^{q'}
    \qquad\text{and}\qquad
    G - F^{q'} = -2F_2^{q'}.
$$
Because $F$ divides $G$, it divides the left hand sides of the above two equations, whence it divides both $F_1$ and $F_2$. Thus, there exist $c_1,c_2\in\bF_{q'^2}$ such that $F_1 = c_1F$ and $F_2 = c_2F$, or equivalently, $M_1 = c_1M$ and $M_2 = c_2M$. If $c_1=c_2=0$, then $M=0$ and there is nothing left to prove. If $c_1\neq 0$ (resp. $c_2\neq 0$), then $M = c_1^{-1}M_1$ (resp. $M = c_2^{-1}M_2$) and the requirement is satisfied.
\end{proof}

\begin{cor}
\label{cor:frobForm_Herm}
Let $X=\{F=0\}\subset\bP^n$ be a reduced Frobenius nonclassical hypersurface over $\bF_q$ of degree $\sqrt{q}+1$. Assume that $\operatorname{char}(\bF_q)\neq 2$ and $F_{1,0} = cF^{\sqrt{q}}$ for some nonzero constant~$c$. Then $X$ is Hermitian.
\end{cor}

\begin{proof}
By Proposition~\ref{prop:F10=powerF}, $F$ is a Frobenius form. Applying Lemma~\ref{lemma:frobForm_Herm_skewHerm} with $q' = \sqrt{q}$ and rescaling $F$ by a constant in $\bF_q$ if necessary, we can express $F$ in the form~\eqref{eqn:frobForm} where $M$ is either Hermitian or skew-Hermitian. If $M$ is Hermitian, then the proof is done. Otherwise, we can pick $c\in\bF_q\backslash \{0\}$ that satisfies $c^{\sqrt{q}} = -c$. Then $cM$ is Hermitian. Since $\{cF = 0\}$ defines the same hypersurface $X$, this completes the proof.
\end{proof}

\section{Upper bounds on degree and characterizations}
\label{sect:degree_upper_bound}

This section is devoted to the proof of Theorem~\ref{mainthm:upper-bound}. We will establish the degree bounds and give characterizations step by step, starting from the simplest case $F_{1,0} = 0$, then the case $d\not\equiv 0\pmod{p}$ where $p = \operatorname{char}(\bF_q)$, and eventually to the full generality. The main machinery involved in the proof is the Koszul complex.

\subsection{Hypersurfaces with vanishing \texorpdfstring{$\mathbf{F_{1,0}}$}{F(1,0)}}
\label{subsect:F10=0}

Let $X=\{F=0\}\subset\bP^n$ be a Frobenius nonclassical hypersurface over $\bF_q$ that satisfies the vanishing condition
$$
    F_{1,0} =
    \sum_{i=0}^{n}x_i^q\frac{\partial F}{\partial x_i} = 0.
$$
Now consider the ring $R\colonequals\bF_q[x_0,\dots,x_n]$, the sequence 
$
    \mathbf{x}^q\colonequals (x_0^q,\dots,x_n^q),
$
and the associated Koszul complex
$$\xymatrix@R=0pt{
    K_\bullet(\mathbf{x}^q):
    0\ar[r] &
    \bigwedge^{n+1}R^{n+1}\ar[r] &
    \cdots\ar[r] &
    \bigwedge^2R^{n+1}\ar[r]^-{\delta_2} &
    R^{n+1}\ar[r]^-{\delta_1} &
    R\ar[r] & 0.\\
    &&&& e_i\ar@{|->}[r] & x_i^q &
}$$
Here $\{e_0,\dots,e_n\}$ is the canonical basis for the free $R$-module $R^{n+1} = \bigoplus_{i=0}^n Re_i$. In this setting, the vanishing of $F_{1,0}$ is equivalent to the condition
\begin{equation}
\label{eqn:F10=0_via_Koszul}
    \sum_{i=0}^{n}\frac{\partial F}{\partial x_i}e_i
    \in\ker\delta_1.
\end{equation}
On the other hand, as the sequence $\mathbf{x}^q = (x_0^q,\dots,x_n^q)$ is regular \cite{Mat89}*{Theorem~16.1}, the complex $K_\bullet(\mathbf{x}^q)$ is exact at degree $i$ for all $i\geq 1$ \cite{Mat89}*{Theorem~16.5~(i)}. In particular,
\begin{equation}
\label{eqn:x-seq_exact_1}
    H_1(K_\bullet(\mathbf{x}^q))
    = \ker\delta_1/\operatorname{im}\delta_{2}
    = 0.
\end{equation}

The following lemma is a consequence of these relations.

\begin{lemma}
\label{lemma:F10=0_exactness}
Retain the setting from above. Then
$$
    \frac{\partial F}{\partial x_j}
    = \sum_{i=0}^nx_{i}^qG_{ij},
    \qquad j=0,\dots,n,
$$
for some $G_{ij}\in\bF_q[x_0,\dots,x_n]$ that satisfies $G_{ji} = -G_{ij}$ and $G_{ii}=0$.
\end{lemma}

\begin{proof}
Relation \eqref{eqn:x-seq_exact_1} implies that the element \eqref{eqn:F10=0_via_Koszul} admits a preimage
$$
    \sum_{i<j}G_{ij}e_i\wedge e_j
    \in\bigwedge^2R^{n+1}
$$
under the differential $\delta_2$. By setting $G_{ij}=-G_{ji}$ for $i > j$ and $G_{ii} = 0$, we get
\begin{align*}
    & \sum_{j=0}^n\left(\frac{\partial F}{\partial x_j}\right)e_j
    = \delta_2\left(\sum_{i<j}G_{ij}e_i\wedge e_j\right)
    = \sum_{i<j}G_{ij}\delta_2(e_i\wedge e_j)
    = \sum_{i<j}G_{ij}(x_i^qe_j - x_j^qe_i) \\
    &= \sum_{i<j}x_i^qG_{ij}e_j - \sum_{j<i}x_i^qG_{ji}e_j
    = \sum_{i<j}x_i^qG_{ij}e_j + \sum_{j<i}x_i^qG_{ij}e_j
    = \sum_{j=0}^n\left(\sum_{i=0}^nx_{i}^qG_{ij}\right)e_{j}.
\end{align*}
Comparing both sides of the equation gives the desired equalities.
\end{proof}

\begin{cor}
\label{cor:F10=0}
Let $X = \{F=0\}\subset\bP^n$ be a Frobenius nonclassical hypersurface of degree~$d$ over $\bF_q$ of characteristic~$p$ which satisfies $F_{1,0} = 0$.
\begin{enumerate}[label=\textup{(\arabic*)}]
    \item\label{F10=0:not-pth-power}
    If $F$ is not a $p$-th power, then $d\geq q+1$. 
    \item\label{F10=0:d-not-0-mod-p}
    If $d\not\equiv 0 \pmod{p}$, then
    $$
        F = \sum_{i,j=0}^nx_{i}^qA_{ij}x_{j}
    $$
    where $A_{ij}$ are polynomials that satisfy $A_{ji} = -A_{ij}$ and $A_{ii}=0$.
\end{enumerate}
\end{cor}

\begin{proof}
The polynomial $F$ is not a $p$-th power if and only if
$$
    \frac{\partial F}{\partial x_j}
    \neq 0
    \quad\text{for some}\quad
    j\in\{0,\dots,n\}.
$$
This implies that some $G_{ij}$ in Lemma~\ref{lemma:F10=0_exactness} is not the zero polynomial. Therefore, $\deg(G_{ij}) = d-1-q \geq 0$. Thus $d\geq q+1$. This proves \ref{F10=0:not-pth-power}. If $d\not\equiv 0\ (\bmod\ p)$, then Euler's formula and Lemma~\ref{lemma:F10=0_exactness} imply that
$$
    F = d^{-1}\sum_{j=0}^n\frac{\partial F}{\partial x_j}x_j
    = d^{-1}\sum_{i,j=0}^nx_{i}^qG_{ij}x_j.
$$
This proves \ref{F10=0:d-not-0-mod-p} by setting $A_{ij}\colonequals d^{-1}G_{ij}$.
\end{proof}

\subsection{Restrictions imposed by smoothness}
\label{subset:restriction_smoothness}

Let $X = \{F = 0\}\subset\bP^n$ be a hypersurface over $\bF_q$ of degree~$d\geq 2$. Consider $R = \bF_q[x_0,\dots,x_n]$ and $R^{n+1} = \bigoplus_{i=0}^nRe_i$ as before. Then the Jacobian ideal
$$
    J_F = \left(
        \frac{\partial F}{\partial x_0},\dots,
        \frac{\partial F}{\partial x_n}
    \right)\subset R
$$
is a proper ideal, and it defines a Koszul complex
$$\xymatrix@R=0pt{
    K_\bullet(J_F):
    0\ar[r] &
    \bigwedge^{n+1}R^{n+1}\ar[r]^-{\delta_{n+1}}&
    \cdots\ar[r]^-{\delta_3} &
    \bigwedge^2R^{n+1}\ar[r]^-{\delta_2} &
    R^{n+1}\ar[r]^-{\delta_1} &
    R\ar[r] & 0. \\
    &&&& e_i\ar@{|->}[r] & \frac{\partial F}{\partial x_i} &
}$$
In this setting, the length of a maximal regular sequence in $J_F$, that is, the \emph{depth} of $J_F$, can be computed by \cite{Mat89}*{Theorem~16.8}
\begin{equation}
\label{eqn:depth-in-homology}
    \operatorname{depth}(J_F)
    = n+1-\max\{
        i \mid H_i(K_\bullet(J_F))
        = \ker\delta_i/\operatorname{im}\delta_{i+1}
        \neq 0
    \}.
\end{equation}
On the other hand, $R$ is a polynomial ring over a field and thus is Cohen--Macaulay \cite{Eis95}*{Proposition~18.9}. This implies that \cite{Eis05}*{Theorem~A2.38}
\begin{equation}
\label{eqn:depth=codim}
    \operatorname{depth}(J_F)
    = \operatorname{codim}(J_F)
\end{equation}
where $\operatorname{codim}(J_F)$ is the Krull codimension of the scheme $\{J_F=0\}\subset\bP^n$.

\begin{lemma}
\label{lemma:smooth_p-nmid-d}
Let $X = \{F = 0\}\subset\bP^n$ be a smooth Frobenius nonclassical hypersurface of degree~$d\geq 2$ over $\bF_q$ of characteristic $p$ such that $d\not\equiv 0 \pmod{p}$. Then
$$
    x_j^q - d^{-1}\left(\frac{F_{1,0}}{F}\right)x_j
    = \sum_{i=0}^n\frac{\partial F}{\partial x_i}\beta_{ij}
    \quad\text{for all}\quad
    j = 0,\dots,n,
$$
where $\beta_{ij}\in\bF_q[x_0,\dots,x_n]$ satisfy $\beta_{ji} = -\beta_{ij}$ and $\beta_{ii}=0$. In particular, we have $d\leq q+1$.
\end{lemma}

\begin{proof}
The assumption $d\not\equiv 0 \pmod{p}$ and Euler's formula imply that the singular locus of $X$ coincides with $\{J_F = 0\}\subset\bP^n$. Hence $\operatorname{codim}(J_F) = n+1$ since $X$ is smooth. Using \eqref{eqn:depth-in-homology} and \eqref{eqn:depth=codim}, we conclude that
$$
    \max\{
        i \mid H_i(K_\bullet(J_F))
        = \ker\delta_i/\operatorname{im}\delta_{i+1}
        \neq 0
    \} = 0.
$$
In particular,
\begin{equation}
\label{eqn:smooth-implies-H1=0}
    H_1(K_\bullet(J_F))
    = \ker\delta_1/\operatorname{im}\delta_{2}
    = 0.
\end{equation}

Define $\alpha\colonequals d^{-1}(F_{1,0}/F) \in \bF_q[x_0,\dots,x_n]$. This is well-defined as $d\not\equiv 0 \pmod{p}$ and $X$ is Frobenius nonclassical. Rearranging the equation gives $F_{1,0} - \alpha dF = 0$, which can be expanded via Euler's formula as
$$
    \sum_{j=0}^n(x_j^q - \alpha x_j)
    \frac{\partial F}{\partial x_j} = 0.
$$
This equation shows that
$$
    \sum_{j=0}^n(x_j^q-\alpha x_j)e_j
    \in\ker\delta_1.
$$
By \eqref{eqn:smooth-implies-H1=0}, this element admits a preimage $\sum_{i<j}\beta_{ij}e_i\wedge e_j\in\bigwedge^2 R^{n+1}$ under $\delta_2$. By setting $\beta_{ij}=-\beta_{ji}$ for $i > j$ and $\beta_{ii} = 0$, we obtain
\begin{align*}
    &\sum_{j=0}^n(x_j^q - \alpha x_j)e_j
    = \delta_2(\sum_{i<j}\beta_{ij}e_i\wedge e_j)
    = \sum_{i<j}\beta_{ij}\delta_2(e_i\wedge e_j)
    = \sum_{i<j}\beta_{ij}\left(
        \frac{\partial F}{\partial x_i}e_j - \frac{\partial F}{\partial x_j}e_i
    \right) \\
    &= \sum_{i<j}\beta_{ij}\frac{\partial F}{\partial x_i}e_j
        - \sum_{j<i}\beta_{ji}\frac{\partial F}{\partial x_i}e_j
    = \sum_{i<j}\beta_{ij}\frac{\partial F}{\partial x_i}e_j
        + \sum_{j<i}\beta_{ij}\frac{\partial F}{\partial x_i}e_j
    = \sum_{j=0}^n\left(
            \sum_{i=0}^n\beta_{ij}\frac{\partial F}{\partial x_i}
        \right)e_j.
\end{align*}
Comparing both sides of the equation gives the desired relations.

To prove that $d\leq q+1$, first note that $x_j^q - x_j\alpha = 0$ implies that $\alpha = x_j^{q-1}$, which cannot hold for all $0\leq j\leq n$. Hence there exists $j$ such that $x_j^q - x_j\alpha \neq 0$. Therefore,
$$
    q = \deg(x_j^q - \alpha x_j)
    = \deg\left(\sum_{i=0}^n\frac{\partial F}{\partial x_i}\beta_{ij}\right)
    = d - 1 + \deg(\beta_{ij}).
$$
This shows that $q\geq d-1$, or equivalently, $d\leq q+1$.
\end{proof}

\begin{cor}
\label{cor:smooth_p-nmid-d_F10=0}
Let $X = \{F = 0\}\subset\bP^n$ be a smooth Frobenius nonclassical hypersurface of degree $d\geq 2$ over $\bF_q$ of characteristic $p$ such that $d\not\equiv 0 \pmod{p}$. Assume additionally that $F_{1,0}=0$. Then there is a nondegenerate skew-symmetric matrix $(A_{ij})$ with entries in $\bF_q$ and zeros along the diagonal such that
$$
    F = \sum_{i,j=0}^nx_{i}^qA_{ij}x_{j}.
$$
In particular, this situation occurs only when $n$ is odd.
\end{cor}

\begin{proof}
Corollary~\ref{cor:F10=0}~\ref{F10=0:not-pth-power} and Lemma~\ref{lemma:smooth_p-nmid-d} imply that $d = q+1$. Corollary~\ref{cor:F10=0}~\ref{F10=0:d-not-0-mod-p} then implies that $F = \sum_{i,j=0}^nx_{i}^qA_{ij}x_{j}$, where $(A_{ij})$ is a skew-symmetric matrix with entries in $\bF_q$ and zeros along the diagonal. The smoothness of $X$ implies that $(A_{ij})$ is nondegenerate since, if not, then there exists a point $[a_0:\cdots:a_n]\in\bP^n(\bF_q)$ such that $\sum_{i=0}^na_iA_{ij} = 0$ for all~$j$, and a straightforward computation shows that this point is a singular point of $X$. In particular, the size of $(A_{ij})$ is even \cite{Lan02}*{XV, Theorem~8.1}. Thus $n$ is odd.
(Note that the matrix $(A_{ij})$ having zeros along the diagonal is essential to ensure that its size is even in the case when $p=2$.)
\end{proof}

\subsection{Characterization in the case of degree \texorpdfstring{$\mathbf{q+1}$}{q+1}}
\label{subsect:degree_q+1}

Lemma~\ref{lemma:smooth_p-nmid-d} can be further refined for the case when degree $d = q+1$. This refinement provides a characterization for smooth Frobenius nonclassical hypersurfaces over $\bF_q$ with this specific degree. This will be stated and proved in Proposition~\ref{prop:smooth_d=q+1} after the next lemma. 

\begin{lemma}
\label{lemma:x^q-ax}
Let $\alpha\in\bF_q[x_0,\dots,x_n]$ be a homogeneous polynomial that is either constantly zero or nonzero of degree $q-1$. Consider the $\bF_q$-vector space
$$
    V = \mathrm{span}\{
        x_j^q-\alpha x_j
        \mid j = 0,\dots,n
    \}\subset\bF_q[x_0,\dots,x_n]
$$
which is of dimension at most $n+1$. Under the situation that $\dim(V)\leq n$, it can only happen that $\dim(V) = n$ and, in this case, there exists a system of coordinates $\{y_0,\dots,y_n\}$ such that $\alpha = y_0^{q-1}$ and that $\{y_j^q-y_0^{q-1}y_j\mid j=1,\dots,n\}$ forms a basis for $V$.
\end{lemma}

\begin{proof}
The condition $\dim(V)\leq n$ means the polynomials $x_j^q - \alpha x_j$, $j = 0,\dots,n$, are linearly dependent, so there exists
$
    (c_0,\dots,c_n)\in \bF_q^{n+1}\setminus\{0\}
$
such that
$$
    \sum_{j=0}^nc_j(x_j^q-\alpha x_j)
    = \sum_{j=0}^nc_jx_j^q - \alpha\sum_{j=0}^nc_jx_j
    = 0.
$$
Since $c_j=c_j^q$ for all $j$, rearranging the above equation gives
$$
    \alpha\sum_{j=0}^nc_jx_j
    = \sum_{j=0}^nc_jx_j^q
    = \sum_{j=0}^nc_j^qx_j^q
    = \left(\sum_{j=0}^nc_jx_j\right)^q.
$$
Eliminating common factors from both sides gives
$$
    \alpha = \left(\sum_{j=0}^nc_jx_j\right)^{q-1}.
$$
As $(c_0,\dots,c_n)$ is a nonzero vector, we can set $y_0=\sum_{j=0}^nc_jx_j$ and complete it to a coordinate system $\{y_0,\dots,y_n\}$. Notice that $\alpha = y_0^{q-1}$ in this setting.

The polynomials $y_i^q - y_0^{q-1}y_i$, $1\leq i\leq n$, cut out the set of $\bF_q$-points away from the hyperplane $\{y_0 = 0\}$, which implies that they are linearly independent over $\bF_q$. To finish the proof, it is sufficient to show that they belong to $V$. Suppose that the transformation between the coordinate systems $\{x_0,\dots,x_n\}$ and $\{y_0,\dots,y_n\}$ is given by
$$
    y_i=\sum_{j=0}^ng_{ij}x_j
    \quad\text{where}\quad
    (g_{ij})\in\mathrm{GL}_{n+1}(\bF_q).
$$
Then $y_i^q=\sum_{j=0}^ng_{ij}x_j^q$. Hence, for $i = 1,\dots,n$,
$$
    y_i^q - \alpha y_i
    = \sum_{j=0}^ng_{ij}x_j^q - \alpha\sum_{j=0}^ng_{ij}x_j
    = \sum_{j=0}^ng_{ij}(x_j^q - \alpha x_j)\in V.
$$
This completes the proof.
\end{proof}

\begin{prop}
\label{prop:smooth_d=q+1}
Let $X = \{F=0\}\subset\bP^n$ be a smooth Frobenius nonclassical hypersurface over $\bF_q$ of degree $d = q+1$ and let $p = \operatorname{char}(\bF_q)$.
\begin{enumerate}[label=\textup{(\arabic*)}]
    \item\label{smooth_d=q+1_n-odd}
    If $n$ is odd, then there exists a nondegenerate skew-symmetric matrix $(A_{ij})$ with entries in $\bF_q$ and zeros along the diagonal such that
    $$
        F = \sum_{i,j=0}^nx_i^qA_{ij}x_j
    $$
    and vice versa. Notice that $F_{1,0}=0$ in this case.
    \item\label{smooth_d=q+1_n-even}
    If $n$ is even, then $p = 2$ and there exists a system of coordinates $\{y_0,\dots,y_n\}$ and a nondegenerate skew-symmetric matrix $(B_{ij})_{1\leq i, j\leq n}$ with entries in $\bF_q$ and zeros along the diagonal such that
    $$
        F_{1,0} = y_0^{q-1}F,
        \qquad
        F = y_0\frac{\partial F}{\partial y_0}
        + \sum_{i,j=1}^ny_i^qB_{ij}y_j,
        \qquad
        \frac{\partial^2 F}{\partial y_0^2} = 0,
    $$
    and vice versa.
\end{enumerate}
\end{prop}

\begin{proof}
Define $\alpha\colonequals F_{1,0}/F$. Lemma~\ref{lemma:smooth_p-nmid-d} together with the assumption $d = q+1$ implies that there exist $\beta_{ij}\in\bF_q$ satisfying $\beta_{ij} = -\beta_{ji}$ and $\beta_{ii} = 0$ such that
\begin{equation}
\label{eqn:x^q-ax_in_partialF}
    x_j^q - \alpha x_j
    = \sum_{i=0}^n\frac{\partial F}{\partial x_i}\beta_{ij},
    \quad j = 0,\dots,n.
\end{equation}
Note that $\frac{\partial F}{\partial x_i}$, $i=0,\dots,n$, are linearly independent over $\bF_q$. Indeed, if not, then these polynomials cut out a nonempty subset in $\bP^n$, which lies on $X$ due to Euler's formula and the assumption $d\not\equiv 0 \pmod{p}$. But this contradicts the hypothesis that $X$ is smooth. On the other hand, the $\bF_q$-vector subspace
$$
    V = \mathrm{span}\{
        x_j^q-\alpha x_j
        \mid j = 0,\dots,n
    \}\subset\bF_q[x_0,\dots,x_n]
$$
has dimension either $n$ or $n+1$ by Lemma~\ref{lemma:x^q-ax}. These properties with \eqref{eqn:x^q-ax_in_partialF} imply that $\mathrm{rank}(\beta_{ij}) = \dim(V)$ which equals either $n$ or $n+1$. Because the rank of a skew-symmetric matrix is always even, we have
$$
    \mathrm{rank}(\beta_{ij})
    = \begin{cases}
        n+1 & \text{if }n\text{ is odd},\\
        n & \text{if }n\text{ is even}.
    \end{cases}
$$

\emph{Assume that $n$ is odd.} Then the matrix $(\beta_{ij})$ is nondegenerate and thus invertible. Define
$$
    (A_{ij})
    \colonequals ((\beta_{ij})^{-1})^t
    = -(\beta_{ij})^{-1}.
$$
Note that $A_{ij} = -A_{ji}$ and $A_{ii} = 0$. From \eqref{eqn:x^q-ax_in_partialF}, we get
$$
    \frac{\partial F}{\partial x_i}
    = \sum_{j=0}^nA_{ij}(x_j^q-\alpha x_j),
    \quad i=0,\dots,n.
$$
Euler's formula then gives
$$
    F = \sum_{i=0}^nx_i\frac{\partial F}{\partial x_i}
    = \sum_{i,j=0}^nx_iA_{ij}(x_j^q-\alpha x_j)
    = \sum_{i,j=0}^nx_iA_{ij}x_j^q
    - \alpha\sum_{i,j=0}^nx_iA_{ij}x_j
    = \sum_{i,j=0}^nx_iA_{ij}x_j^q.
$$
This proves one implication of \ref{smooth_d=q+1_n-odd}. The converse follows from the fact that, as explained in the proof of Corollary~\ref{cor:smooth_p-nmid-d_F10=0}, the matrix $(A_{ij})$ being nondegenerate implies that $n$ is odd. The vanishing $F_{1,0} = 0$ can be verified directly from the formula of $F$.

\emph{Assume that $n$ is even.} The fact that $\dim(V) = n$ and Lemma~\ref{lemma:x^q-ax} imply that there exists a system of coordinates $\{y_0,\dots,y_n\}$ such that $\alpha = y_0^{q-1}$. In particular, the first equation in \ref{smooth_d=q+1_n-even} holds. Recall that the property of being Frobenius nonclassical is invariant under a change of coordinates over $\bF_q$. Therefore, we can apply Lemma~\ref{lemma:smooth_p-nmid-d} in the new coordinates, which asserts that there exist $\eta_{ij}\in\bF_q$ satisfying $\eta_{ij} = -\eta_{ji}$ and $\eta_{ii} = 0$ such that
\begin{equation}
\label{eqn:y^q-ay_in_partialF}
    y_j^q - y_0^{q-1}y_j
    = \sum_{i=0}^n\frac{\partial F}{\partial y_i}\eta_{ij},
    \quad j = 0,\dots,n.
\end{equation}
For $j = 0$, the above relation reduces to $0 = \sum_{i=0}^n\frac{\partial F}{\partial y_i}\eta_{i0}$. As $\frac{\partial F}{\partial y_i}$, $i=0,\dots,n$, are linearly independent over $\bF_q$, we conclude that $\eta_{i0} = -\eta_{0i} = 0$ for $i = 0,\dots,n$. It follows that the minor $(\eta_{ij})_{1\leq i,j\leq n}$ has full rank $n$ and thus is invertible. Define
$$
    (B_{ij})_{1\leq i,j\leq n}
    \colonequals((\eta_{ij})_{1\leq i,j\leq n}^{-1})^t
    = -(\eta_{ij})_{1\leq i,j\leq n}^{-1}.
$$
Then $B_{ij} = -B_{ji}$ and $B_{ii} = 0$. Collecting the relations in \eqref{eqn:y^q-ay_in_partialF} for $j=1,\dots,n$, we get
$$
    \frac{\partial F}{\partial y_i}
    = \sum_{j=1}^nB_{ij}(y_j^q-y_0^{q-1}y_j),
    \quad i=1,\dots,n.
$$
Applying Euler's formula, we get
\begin{equation}
\label{eqn:F-y0partialF}
\begin{aligned}
    F - y_0\frac{\partial F}{\partial y_0}
    = \sum_{i=1}^ny_i\frac{\partial F}{\partial y_i}
    &= \sum_{i,j=1}^ny_iB_{ij}(y_j^q-y_0^{q-1}y_j) \\
    &= \sum_{i,j=1}^ny_iB_{ij}y_j^q
    - y_0^{q-1}\sum_{i,j=1}^ny_iB_{ij}y_j
    = \sum_{i,j=1}^ny_iB_{ij}y_j^q.
\end{aligned}
\end{equation}
This proves the second equation in \ref{smooth_d=q+1_n-even}. Applying $\frac{\partial}{\partial y_0}$ to both sides of \eqref{eqn:F-y0partialF} gives
\begin{equation}
\label{eqn:y0-partial-order-2}
    y_0\frac{\partial^2 F}{\partial y_0^2} = 0,
    \qquad\text{whence}\qquad
    \frac{\partial^2 F}{\partial y_0^2} = 0.
\end{equation}
This proves the third equation in \ref{smooth_d=q+1_n-even}.

Now we prove that $p = 2$. For the sake of simplicity, we denote
$$
    G\colonequals\frac{\partial F}{\partial y_0}
    \qquad\text{and}\qquad
    H\colonequals\sum_{i,j=1}^ny_iB_{ij}y_j^q.
$$
Notice that $H_{1,0} = 0$. Rewrite \eqref{eqn:F-y0partialF} as $F = y_0G + H$. Then a straightforward computation gives
$$
    F_{1,0}
    = y_0^qG + y_0G_{1,0} + H_{1,0}
    = y_0^qG + y_0G_{1,0}.
$$
It follows that
$$
    y_0^qG + y_0G_{1,0}
    = F_{1,0}
    = y_0^{q-1}F
    = y_0^{q-1}(y_0G + H)
    = y_0^qG + y_0^{q-1}H.
$$
Rearranging the terms and eliminating common factors to get
\begin{equation}
\label{eqn:G10_H}
    G_{1,0} = y_0^{q-2}H.
\end{equation}
Recall from \eqref{eqn:y0-partial-order-2} that $\frac{\partial G}{\partial y_0} = 0$. Applying $\frac{\partial}{\partial y_0}$ to \eqref{eqn:G10_H}, the left hand side gives
$$
    \frac{\partial G_{1,0}}{\partial y_0}
    = \frac{\partial }{\partial y_0}\left(
        \sum_{i=0}^ny_i^q\frac{\partial G}{\partial y_i}
    \right)
    = \sum_{i=0}^ny_i^q\frac{\partial^2 G}{\partial y_0\partial y_i}
    = \sum_{i=0}^ny_i^q\frac{\partial^2 G}{\partial y_i\partial y_0}
    = 0,
$$
while the right hand side with the fact that $\frac{\partial H}{\partial y_0} = 0$ gives
$$
    \frac{\partial}{\partial y_0}(y_0^{q-2}H)
    = -2y_0^{q-3}H + y_0^{q-2}\frac{\partial H}{\partial y_0}
    = -2y_0^{q-3}H.
$$
Hence $0 = -2y_0^{q-3}H$. We have $H\neq 0$ since $(B_{ij})_{1\leq i,j\leq n}$ is nondegenerate, so $p=2$. This completes the proof of one implication in \ref{smooth_d=q+1_n-even}.

The converse of \ref{smooth_d=q+1_n-even} holds because the skew-symmetric matrix $(B_{ij})$ with zeros along the diagonal is nondegenerate only when its size $n$ is even.
\end{proof}

\subsection{The upper bound \texorpdfstring{$\mathbf{d\leq q+2}$}{d<=q+2} and further characterization}
\label{subsect:degree_q+2}

Let us proceed to study smooth Frobenius nonclassical hypersurfaces $X=\{F=0\}\subset\bP^n$ of degree~$d\geq 2$ over~$\bF_q$ without the assumption that $d$ is not divisible by $p = \operatorname{char}(\bF_q)$. To deal with this general case, we work with the quotient ring $\overline{R}\colonequals\bF_q[x_0,\dots,x_n]/(F)$ and the ideal
$$
    \overline{J_F}\colonequals\left(
        \frac{\partial F}{\partial x_0},\dots,
        \frac{\partial F}{\partial x_n}
    \right)
    \subset\overline{R}.
$$
Note that $\overline{J_F}\neq\overline{R}$ due to $d\geq 2$. This ideal determines a Koszul complex
$$\xymatrix@R=0pt{
    K_\bullet(\overline{J_F}):
    0\ar[r] &
    \bigwedge^{(n+1)}\overline{R}^{n+1}\ar[r]^-{\delta_{n+1}}&
    \cdots\ar[r]^-{\delta_3} &
    \bigwedge^2\overline{R}^{n+1}\ar[r]^-{\delta_2} &
    \overline{R}^{n+1}\ar[r]^-{\delta_1} &
    \overline{R}\ar[r] & 0.\\
    &&&& e_i\ar@{|->}[r] & \frac{\partial F}{\partial x_i}
}$$
The fact that $X\subset\bP^n$ is a hypersurface implies that $\overline{R}$ is Cohen--Macaulay \cite{Eis95}*{Proposition~18.13}. By \cite{Eis05}*{Theorem~A2.38} and \cite{Mat89}*{Theorem~16.8},
$$
    \operatorname{codim}(\overline{J_F})
    = \operatorname{depth}(\overline{J_F})
    = n+1-\max\left\{
        i \,\middle\vert\, H_i(K_\bullet(\overline{J_F}))
        = \ker\delta_i/\operatorname{im}\delta_{i+1}
        \neq 0
    \right\}.
$$
The smoothness of $X$ implies that $\operatorname{codim}(\overline{J_F})=n$, so the above relation reduces to
$$
    \max\left\{
        i \,\middle\vert\, H_i(K_\bullet(\overline{J_F}))
        = \ker\delta_i/\operatorname{im}\delta_{i+1}
        \neq 0
    \right\}
    = 1.
$$
In particular,
\begin{equation}
\label{eqn:H2=0}
    H_2(K_\bullet(\overline{J_F}))
    = \ker\delta_2/\operatorname{Im}\delta_3
    = 0.
\end{equation}

\begin{lemma}
\label{lemma:vanishingH2}
Let $X=\{F=0\}\subset\bP^n$ be a smooth Frobenius nonclassical hypersurface of degree~$d\geq 2$ over $\bF_q$. Then $d\leq q+2$ and there exists
$$
    \gamma = \sum_{i<j<k}\gamma_{ijk}e_i\wedge e_j\wedge e_k
    \in\bigwedge^3\overline{R}^{n+1}
$$
such that
$$
    x_ix_j^q-x_i^qx_j
    = \sum_{k=0}^n\gamma_{ijk}\frac{\partial F}{\partial x_k}\pmod{F},
    \qquad 0\leq i<j\leq n.
$$
\end{lemma}

In the last equation above, we apply the convention that
\begin{equation}
\label{eqn:conven_gamma}
    \gamma_{\sigma(i)\sigma(j)\sigma(k)}
    = \mathrm{sgn}(\sigma)\gamma_{ijk}
\end{equation}
for any permutation $\sigma$ on $i,j,k$. In particular, $\gamma_{ijk} = 0$ if any two of $i,j,k$ are identical. We will continue to use this convention in the remaining part of this section.

\begin{proof}
Let us consider the two elements in $\overline{R}^{n+1}$
$$
    \mathbf{x}\colonequals\sum_{i=0}^nx_ie_i
    \qquad\text{and}\qquad
    \mathbf{x}^{\,q}\colonequals\sum_{i=0}^nx_i^qe_i.
$$
Euler's formula implies that $\mathbf{x}\in\ker\delta_1$. On the other hand, $X$ is Frobenius nonclassical and so $\mathbf{x}^{\,q}\in\ker\delta_1$. These facts imply that
$$
    \delta_2(\mathbf{x}\wedge\mathbf{x}^{\,q})
    = \left(\sum_{i=0}^nx_i\frac{\partial F}{\partial x_i}\right)\mathbf{x}^{\,q}
    - \left(\sum_{i=0}^nx_i^q\frac{\partial F}{\partial x_i}\right)\mathbf{x}
    = \delta_1(\mathbf{x})\mathbf{x}^{\,q}
    - \delta_1(\mathbf{x}^{\,q})\mathbf{x}
    = 0.
$$
That is, $\beta\colonequals\mathbf{x}\wedge\mathbf{x}^{\,q}\in\ker\delta_2$. By \eqref{eqn:H2=0}, there exists
$$
    \gamma = \sum_{i<j<k}\gamma_{ijk}e_i\wedge e_j\wedge e_k
    \in\bigwedge^3\overline{R}^{n+1}
$$
such that $\beta = \delta_3(\gamma)$.

For the sake of simplicity, we denote $F_i\colonequals\frac{\partial F}{\partial x_i}$ for $i=0,\dots,n$. Expanding both sides of the equation $\beta = \delta_3(\gamma)$ gives
\begin{align*}
    \beta = \sum_{i<j}(x_ix_j^q-x_i^qx_j)e_i\wedge e_j
    &= \sum_{i<j<k}\gamma_{ijk}\delta_3(e_i\wedge e_j\wedge e_k)\\
    &= \sum_{i<j<k}\gamma_{ijk}\left(
        F_ie_j\wedge e_k
        - F_je_i\wedge e_k
        + F_ke_i\wedge e_j
    \right)\\
    &= \sum_{i<j<k}\gamma_{ijk}F_ie_j\wedge e_k
        - \sum_{i<j<k}\gamma_{ijk}F_je_i\wedge e_k
        + \sum_{i<j<k}\gamma_{ijk}F_ke_i\wedge e_j\\
    &= \sum_{k<i<j}\gamma_{kij}F_ke_i\wedge e_j
        - \sum_{i<k<j}\gamma_{ikj}F_ke_i\wedge e_j
        + \sum_{i<j<k}\gamma_{ijk}F_ke_i\wedge e_j\\
    &= \sum_{k<i<j}\gamma_{ijk}F_ke_i\wedge e_j
        + \sum_{i<k<j}\gamma_{ijk}F_ke_i\wedge e_j
        + \sum_{i<j<k}\gamma_{ijk}F_ke_i\wedge e_j\\
    &= \sum_{i<j}\left(
            \sum_{k=0}^n\gamma_{ijk}F_k
        \right)
        e_i\wedge e_j.
\end{align*}
Comparing both sides of the equality gives the desired relations
\begin{equation}
\label{eqn:alpha_in_beta}
    x_ix_j^q-x_i^qx_j
    = \sum_{k=0}^n\gamma_{ijk}F_k\pmod{F},
    \qquad 0\leq i<j\leq n.
\end{equation}

Let us prove that $d\leq q+2$. Assume, to the contrary, that $d>q+2$. Then the polynomial on the left hand side of equation~\eqref{eqn:alpha_in_beta} has degree $q+1 < d$. This implies that the polynomial on the right hand side has degree $q+1$ as well, and the equation holds without modulo $F$. The assumption $d>q+2$ also implies that $\deg(F_k) = d-1 > q+1$. But this implies that the right hand side of \eqref{eqn:alpha_in_beta} has degree strictly greater than $q+1$, which contradicts to our previous conclusion. Therefore, we must have $d\leq q+2$.
\end{proof}

\begin{prop}
\label{prop:smooth_d=q+2}
Let $X=\{F=0\}\subset\bP^n$, where $n\geq 2$, be a smooth Frobenius nonclassical hypersurface of degree $d$ over $\bF_q$ of characteristic $p$ such that $d = q+2$. Then $p = n = 2$ and, upon rescaling $F$ by a nonzero constant, we have
$$
    \frac{\partial F}{\partial x_0} = x_1x_2^q - x_1^qx_2,
    \qquad
    \frac{\partial F}{\partial x_1} = x_2x_0^q - x_2^qx_0,
    \qquad
    \frac{\partial F}{\partial x_2} = x_0x_1^q - x_0^qx_1.
$$
In particular, we have $F_{1,0}=0$ and
$$
    F = x_0x_1x_2(x_0^{q-1} + x_1^{q-1} + x_2^{q-1}) + G(x_0^2,x_1^2,x_2^2)
$$
for some polynomial $G$.
\end{prop}

\begin{proof}
First of all, if $p\neq 2$, then $d = q+2\equiv 2\not\equiv 0\ (\bmod\ p)$. But this forces $d\leq q+1$ by Lemma~\ref{lemma:smooth_p-nmid-d}, a contradiction. Hence $p = 2$.

Let us show that $n = 2$. By hypothesis, $\deg(F) = d = q+2 > q+1$. This implies that the relations from Lemma~\ref{lemma:vanishingH2}:
\begin{equation}
\label{eqn:x_ix_j^q-x_i^qx_j_in_partialF}
    x_ix_j^q-x_i^qx_j
    = \sum_{k=0}^n\gamma_{ijk}\frac{\partial F}{\partial x_k}
\end{equation}
hold without modulo $F$ and that $\gamma_{ijk}\in\bF_q$. The $n(n+1)/2$ polynomials $x_ix_j^q-x_i^qx_j$ for $0\leq i<j\leq n$ are linearly independent over $\bF_q$. On the other hand, the polynomials $\frac{\partial F}{\partial x_k}$, $0\leq k\leq n$, span an $\bF_q$-vector subspace of dimension at most $n+1$ in $\bF_q[x_0,\dots,x_n]$. These two facts, together with \eqref{eqn:x_ix_j^q-x_i^qx_j_in_partialF}, imply that
$$
    \frac{n(n+1)}{2} \leq n+1,
    \qquad\text{or equivalently,}\qquad
    n\leq 2.
$$
Since $n\geq 2$ by hypothesis, we must have $n = 2$.

By convention~\eqref{eqn:conven_gamma}, $\gamma_{ijk}=0$ if any two of $i,j,k$ coincide. Hence, \eqref{eqn:x_ix_j^q-x_i^qx_j_in_partialF} under the condition that $n=2$ provides three nontrivial relations
$$
    \gamma_{120}\cdot\frac{\partial F}{\partial x_0}
    = x_1x_2^q - x_1^qx_2,
    \quad
    \gamma_{201}\cdot\frac{\partial F}{\partial x_1}
    = x_2x_0^q - x_2^qx_0,
    \quad
    \gamma_{012}\cdot\frac{\partial F}{\partial x_2}
    = x_0x_1^q - x_0^qx_1.
$$
By the same convention, $\gamma_{120}=\gamma_{201}=\gamma_{012}$, so we can denote these elements simultaneously as $c$. Note that $c\neq 0$. Replacing $F$ by $cF$ then gives the desired expressions for the partial derivatives. The vanishing of $F_{1,0}$ and the formula for $F$ are straightforward computations using these expressions.
\end{proof}

\begin{cor}
\label{cor:d=q+2_no-Fq-pt}
Let $X=\{F=0\}\subset\bP^n$, where $n\geq 2$, be a smooth Frobenius nonclassical hypersurface over $\bF_q$ of degree $d = q+2$. Then $X$ contains no $\bF_q$-point.
\end{cor}

\begin{proof}
By Proposition~\ref{prop:smooth_d=q+2}, the partial derivatives of $F$ vanish at every $\bF_q$-point. If $X$ contains any $\bF_q$-point, then the point has to be singular, which cannot happen as we assume $X$ to be smooth.
\end{proof}

\subsection{Sharpness of the upper bounds}
\label{subsect:proof_upper_bounds}

Example~\ref{eg:d=q+1_n-odd} shows that the upper bound in Lemma~\ref{lemma:smooth_p-nmid-d} is sharp. In the following, we exhibit more examples of smooth Frobenius nonclassical hypersurfaces whose degrees attain the upper bounds in Lemmas~\ref{lemma:smooth_p-nmid-d} and \ref{lemma:vanishingH2}. Then we finish the proof of Theorem~\ref{mainthm:upper-bound} at the end of this section.

\begin{eg}
\label{eg:d=q+1_n-even}
Over $\bF_4 = \bF_2(a)$, where $a^2+a+1=0$, the plane curve $X\subset\bP^2$ defined by
$$
    F = x\left(
        (a+1)x^4 + x^2y^2 + x^2yz + x^2z^2 + y^4 + y^2z^2 + z^4
    \right)
    + y^4z + yz^4
$$
provides an example for Proposition~\ref{prop:smooth_d=q+1}~\ref{smooth_d=q+1_n-even} as well as Theorem~\ref{mainthm:upper-bound}~\ref{main:d=q+1_n-even}. The Hessian matrix of this example equals
$$
\left(
    \frac{\partial^2 F}{\partial x_i\partial x_j}
\right)
= \begin{pmatrix}
    0 & x^2z & x^2y \\
    x^2z & 0 & x^3 \\
    x^2y & x^3 & 0
\end{pmatrix}
$$
which has a zero determinant. In general, the determinant of the Hessian matrix of a smooth Frobenius nonclassical hypersurface $X$ vanishes along $X$ by \cite{ADL21}*{Lemma~4.6}. But the Hessian matrix itself may not vanish along $X$ due to this example.
\end{eg}

\begin{eg}
\label{eg:d=q+2}
Over $\bF_2$, the plane curve $X\subset\bP^2$ defined by
$$
    (x+y+z)xyz+x^2y^2+x^2z^2+y^2z^2+x^4+y^4+z^4 = 0
$$
provides an example for Proposition~\ref{prop:smooth_d=q+2} as well as Theorem~\ref{mainthm:upper-bound}~\ref{main:d=q+2}. This is the Dickson--Guralnick--Zieve curve for $q=2$ \cite{GKT19}. As another example, we have the plane curve over $\bF_4$ defined by
$$
    (x^3+y^3+z^3)xyz + x^2 y^2(x^2+y^2) + y^2 z^2(y^2+z^2)+ (a+1)x^6 + y^6+ az^6 = 0
$$
where $a$ is a multiplicative generator of $\bF_4^{\ast}$ that satisfies $a^2+a+1=0$. Finally, let $a$ be a multiplicative generator of $\bF_8^{\ast}$ satisfying $a^3+a+1=0$. Then the plane curve defined by
\begin{align*}
    &(x^7+y^7+z^7)xyz + a^{2} x^{10}+x^{6} y^{4}+x^{4} y^{6}+(a+1) y^{10}+(a^{2}+1) x^{8} z^{2}+x^{4} y^{4} z^{2} \\ & + a y^{8} z^{2} +x^{6} z^{4}+x^{2}
      y^{4} z^{4}+y^{6} z^{4}+x^{4} z^{6}+y^{4} z^{6}+(a^{2}+1) x^{2} z^{8}+a y^{2} z^{8}+a z^{10} = 0
\end{align*}
provides an example over $\bF_8$.
\end{eg}

\begin{proof}[Proof of Theorem~\ref{mainthm:upper-bound}]
The bound $d\leq q+2$ is a consequence of Lemma~\ref{lemma:vanishingH2}. The ``only if'' parts of \ref{main:d=q+1_n-odd} and \ref{main:d=q+1_n-even} follow respectively from Proposition~\ref{prop:smooth_d=q+1}~\ref{smooth_d=q+1_n-odd} and \ref{smooth_d=q+1_n-even}. The ``only if'' part of \ref{main:d=q+2} follows from Proposition~\ref{prop:smooth_d=q+2}. For the converse of \ref{main:d=q+1_n-odd}, we have $d=q+1$ directly from the formula. The integer $n$ is odd because the matrix $(A_{ij})$, which has size~$n+1$, is skew-symmetric with zeros along its diagonal and is nondegenerate. The arguments for the converses of \ref{main:d=q+1_n-even} and \ref{main:d=q+2} are similar.

In \ref{main:d=q+1_n-odd} and \ref{main:d=q+2}, straightforward computations give $F_{1,0} = 0$. Conversely, let us assume $F_{1,0} = 0$. Then Corollary~\ref{cor:F10=0}~\ref{F10=0:not-pth-power} and the fact that $d\leq q+2$ imply $d=q+1$ or $d=q+2$. If $d=q+1$, then we are in case~\ref{main:d=q+1_n-odd} by Corollary~\ref{cor:smooth_p-nmid-d_F10=0}. If $d=q+2$, then case~\ref{main:d=q+2} occurs by Proposition~\ref{prop:smooth_d=q+2}.
\end{proof}

\section{Frobenius nonclassical hypersurfaces with separated variables}
\label{sect:sep-var}

In this section, we prove that a smooth Frobenius nonclassical hypersurface over $\bF_q$ with separated variables has degree $1\pmod{p}$ where $p = \operatorname{char}(\bF_q)$. Let us first show that examples in Theorem~\ref{mainthm:upper-bound}~\ref{main:d=q+2} do \emph{not} have separated variables.

\begin{lemma}
\label{lemma:d=q+2_not_sep_var}
Smooth Frobenius nonclassical hypersurfaces over $\bF_q$ of degree $d = q+2$ do not have separated variables.
\end{lemma}

\begin{proof}
Let $X = \{F = 0\}\subset\bP^2$ be such an example. Assume, to the contrary, that it has separated variables. Then there exists a system of coordinates $\{x_0,x_1,x_2\}$, a polynomial $H = H(x_0,x_1)$, and a constant $c\neq 0$ such that
$$
    F = H(x_0,x_1) + cx_2^{q+2}.
$$
However, this implies $\frac{\partial F}{\partial x_2} = 0$, contradicting Proposition~\ref{prop:smooth_d=q+2}.
\end{proof}

As preparation for the proof of Theorem~\ref{mainthm:sep_var_d=1-mod-p}, let us recall an elementary fact about polynomial arithmetic.

\begin{lemma}
\label{lemma:poly-arith}
Let $k$ be an arbitrary field, $g,h\in k$ be nonzero constants, $r(t)\in k[t]$ be a polynomial of degree $m$, and $d$ be a positive integer. Suppose that
$$
    (gt^d + h)\cdot r(t) = at^{d+m} + b
$$
for some nonzero $a,b\in k$. Then $d$ divides $m$.
\end{lemma}

\begin{proof}[Proof sketch.]
The assertion can be verified by writing $r(t) = \sum_{i=0}^{m}r_it^i$ and then comparing the coefficients on both sides of the equation. We leave the details to the reader.
\end{proof}

\begin{proof}[Proof of Theorem~\ref{mainthm:sep_var_d=1-mod-p}]
If $F_{1,0} = 0$, then we are in case~\ref{main:d=q+1_n-odd} or \ref{main:d=q+2} of Theorem~\ref{mainthm:upper-bound}. If case~\ref{main:d=q+1_n-odd} occurs, then $d = q+1 \equiv 1\pmod{p}$. On the other hand, case~\ref{main:d=q+2} does not occur due to Lemma~\ref{lemma:d=q+2_not_sep_var}. This shows that the statement holds when $F_{1,0} = 0$.

Assume that $F_{1,0} \neq 0$. If $d = 1$, then there is nothing to prove, so we assume $d\geq 2$ in the rest of the proof. By hypothesis, there exists a coordinate system $\{x_0, \dots, x_n\}$ and an integer $m\in\{0,\dots,n-1\}$ such that
$$
	F(x_0, \dots, x_n)
    = G(x_0, \dots, x_m) + H(x_{m+1},\dots, x_n).
$$
The assumption that $d\geq 2$, along with the smoothness of $X$, ensures that $G$ and $H$ are not constantly zero. The assumption that $X$ is Frobenius nonclassical implies that there exists a polynomial $R$ such that $FR = F_{1,0}$, or equivalently,
\begin{equation}
\label{eqn:(G+H)R=G10+H10}
	(G+H)\cdot R = G_{1,0}+H_{1,0}.
\end{equation}

Let us extend $\bF_q$ to a function field
$
    k\colonequals
    \bF_q(u_0,\dots,u_n)
$
where $u_0,\dots,u_n$ are formal variables. Consider the line $\ell\subset\bP^n_k$ spanned by the points
$$
    [u_0:\cdots:u_m:0:\cdots:0]
    \qquad\text{and}\qquad
    [0:\cdots:0:u_{m+1}:\cdots:u_n],
$$
which can be expressed by the parametric equations with affine parameter $t$:
$$
    x_i = \begin{cases}
        u_it &\quad\text{for}\quad i = 0,\dots,m,\\
        u_i &\quad\text{for}\quad i= m+1,\dots,n.
    \end{cases}
$$
The restriction of $(G+H)$ to $\ell$ has the form
\begin{align*}
    (G+H)|_\ell
    &= G(u_0t,\dots,u_mt) + H(u_{m+1},\dots,u_n)\\
    &= G(u_0,\dots,u_m)t^d + H(u_{m+1},\dots,u_n)
    = gt^d + h
    \quad\text{where}\quad
    g, h\in k\setminus\{0\}.
\end{align*}
Similarly, $(G_{1,0}+H_{1,0})|_\ell = at^{d+q-1}+b$ for some $a, b\in k$. Hence, restricting \eqref{eqn:(G+H)R=G10+H10} to $\ell$ gives
$$
	(gt^d+h)\cdot r(t)
	= at^{d+q-1}+b
	\qquad\text{where}\qquad
    r(t) = R|_\ell\in k[t].
$$
The assumption $F_{1,0}\neq 0$ implies that $r(t)\neq 0$. Together with the fact that $g,h\neq 0$, one can verify that $a,b\neq 0$. Lemma~\ref{lemma:poly-arith} then implies that $d$ divides $q-1$.

According to \cite{Kle86}*{page 191}, the fact that $X$ is nonreflexive \cite{ADL21}*{Theorem~4.5} implies that $d\equiv 0$ or $1\pmod{p}$. If $d\equiv 0\pmod{p}$, then the fact that $d$ divides $q-1$ implies that $p$ divides $1$, a contradiction. Therefore, we must have $d\equiv 1 \pmod{p}$.
\end{proof}



\bibliographystyle{alpha}
\bibliography{Transversality}

\begin{bibdiv}
\begin{biblist}

\bib{ADL21}{article}{
      author={Asgarli, Shamil},
      author={Duan, Lian},
      author={Lai, Kuan-Wen},
       title={Transverse lines to surfaces over finite fields},
        date={2021},
        ISSN={0025-2611},
     journal={Manuscripta Math.},
      volume={165},
      number={1-2},
       pages={135\ndash 157},
         url={https://doi.org/10.1007/s00229-020-01200-7},
      review={\MR{4242565}},
}

\bib{ADL24hyper}{article}{
      author={Asgarli, Shamil},
      author={Duan, Lian},
      author={Lai, Kuan-Wen},
       title={Transverse linear subspaces to hypersurfaces over finite fields},
        date={2024},
        ISSN={1071-5797,1090-2465},
     journal={Finite Fields Appl.},
      volume={95},
       pages={Paper No. 102396, 25},
         url={https://doi.org/10.1016/j.ffa.2024.102396},
      review={\MR{4711025}},
}

\bib{AGY-blocking-1}{article}{
      author={Asgarli, Shamil},
      author={Ghioca, Dragos},
      author={Yip, Chi~Hoi},
       title={Plane curves giving rise to blocking sets over finite fields},
        date={2023},
        ISSN={0925-1022,1573-7586},
     journal={Des. Codes Cryptogr.},
      volume={91},
      number={11},
       pages={3643\ndash 3669},
         url={https://doi.org/10.1007/s10623-023-01264-y},
      review={\MR{4659423}},
}

\bib{AGY-blocking-2}{article}{
      author={Asgarli, Shamil},
      author={Ghioca, Dragos},
      author={Yip, Chi~Hoi},
       title={Most plane curves over finite fields are not blocking},
        date={2024},
        ISSN={0097-3165,1096-0899},
     journal={J. Combin. Theory Ser. A},
      volume={204},
       pages={Paper No. 105871, 26},
         url={https://doi.org/10.1016/j.jcta.2024.105871},
      review={\MR{4703360}},
}

\bib{BC66}{article}{
      author={Bose, R.~C.},
      author={Chakravarti, I.~M.},
       title={Hermitian varieties in a finite projective space {${\rm
  PG}(N,\,q^{2})$}},
        date={1966},
        ISSN={0008-414X},
     journal={Canadian J. Math.},
      volume={18},
       pages={1161\ndash 1182},
         url={https://doi.org/10.4153/CJM-1966-116-0},
      review={\MR{200782}},
}

\bib{BH17}{article}{
      author={Borges, Herivelto},
      author={Homma, Masaaki},
       title={Points on singular {F}robenius nonclassical curves},
        date={2017},
        ISSN={1678-7544},
     journal={Bull. Braz. Math. Soc. (N.S.)},
      volume={48},
      number={1},
       pages={93\ndash 101},
         url={https://doi.org/10.1007/s00574-016-0008-6},
      review={\MR{3623764}},
}

\bib{BMT14}{article}{
      author={Borges, Herivelto},
      author={Motta, Beatriz},
      author={Torres, Fernando},
       title={Complete arcs arising from a generalization of the {H}ermitian
  curve},
        date={2014},
        ISSN={0065-1036},
     journal={Acta Arith.},
      volume={164},
      number={2},
       pages={101\ndash 118},
         url={https://doi.org/10.4064/aa164-2-1},
      review={\MR{3224828}},
}

\bib{Borges-Arcs}{article}{
      author={Borges, Herivelto},
       title={On complete {$(N,d)$}-arcs derived from plane curves},
        date={2009},
        ISSN={1071-5797},
     journal={Finite Fields Appl.},
      volume={15},
      number={1},
       pages={82\ndash 96},
         url={https://doi.org/10.1016/j.ffa.2008.08.003},
      review={\MR{2468994}},
}

\bib{Bor09}{article}{
      author={Borges, Herivelto},
       title={On multi-{F}robenius non-classical plane curves},
        date={2009},
        ISSN={0003-889X},
     journal={Arch. Math. (Basel)},
      volume={93},
      number={6},
       pages={541\ndash 553},
         url={https://doi.org/10.1007/s00013-009-0056-x},
      review={\MR{2574788}},
}

\bib{Bor16}{article}{
      author={Borges, Herivelto},
       title={Frobenius nonclassical components of curves with separated
  variables},
        date={2016},
        ISSN={0022-314X},
     journal={J. Number Theory},
      volume={159},
       pages={402\ndash 425},
         url={https://doi.org/10.1016/j.jnt.2015.07.006},
      review={\MR{3412730}},
}

\bib{Eis05}{book}{
      author={Eisenbud, David},
       title={The geometry of syzygies},
      series={Graduate Texts in Mathematics},
   publisher={Springer-Verlag, New York},
        date={2005},
      volume={229},
        ISBN={0-387-22215-4},
        note={A second course in commutative algebra and algebraic geometry},
      review={\MR{2103875}},
}

\bib{Eis95}{book}{
      author={Eisenbud, David},
       title={Commutative algebra. {W}ith a view toward algebraic geometry},
      series={Graduate Texts in Mathematics},
   publisher={Springer-Verlag, New York},
        date={1995},
      volume={150},
        ISBN={0-387-94268-8; 0-387-94269-6},
         url={https://doi.org/10.1007/978-1-4612-5350-1},
      review={\MR{1322960}},
}

\bib{GKT19}{article}{
      author={Giulietti, Massimo},
      author={Korchm\'{a}ros, G\'{a}bor},
      author={Timpanella, Marco},
       title={On the {D}ickson-{G}uralnick-{Z}ieve curve},
        date={2019},
        ISSN={0022-314X},
     journal={J. Number Theory},
      volume={196},
       pages={114\ndash 138},
  url={https://doi-org.ezproxy.library.ubc.ca/10.1016/j.jnt.2018.09.020},
      review={\MR{3906470}},
}

\bib{GPTU02}{article}{
      author={Giulietti, Massimo},
      author={Pambianco, Fernanda},
      author={Torres, Fernando},
      author={Ughi, Emanuela},
       title={On complete arcs arising from plane curves},
        date={2002},
        ISSN={0925-1022},
     journal={Des. Codes Cryptogr.},
      volume={25},
      number={3},
       pages={237\ndash 246},
         url={https://doi.org/10.1023/A:1014979211916},
      review={\MR{1900770}},
}

\bib{Hei96}{article}{
      author={Heim, Udo},
       title={Blocking sets in finite projective spaces},
    language={German},
        date={1996},
        ISSN={0373-8221},
     journal={Mitt. Math. Semin. Gie{{\ss}}en},
      volume={226},
       pages={4\ndash 82},
}

\bib{HK13_bound}{article}{
      author={Homma, Masaaki},
      author={Kim, Seon~Jeong},
       title={An elementary bound for the number of points of a hypersurface
  over a finite field},
        date={2013},
        ISSN={1071-5797},
     journal={Finite Fields Appl.},
      volume={20},
       pages={76\ndash 83},
         url={https://doi.org/10.1016/j.ffa.2012.11.002},
      review={\MR{3015353}},
}

\bib{HK16}{article}{
      author={Homma, Masaaki},
      author={Kim, Seon~Jeong},
       title={The characterization of {H}ermitian surfaces by the number of
  points},
        date={2016},
        ISSN={0047-2468},
     journal={J. Geom.},
      volume={107},
      number={3},
       pages={509\ndash 521},
         url={https://doi.org/10.1007/s00022-015-0283-1},
      review={\MR{3563207}},
}

\bib{HKT08}{book}{
      author={Hirschfeld, J. W.~P.},
      author={Korchm\'{a}ros, G.},
      author={Torres, F.},
       title={Algebraic curves over a finite field},
      series={Princeton Series in Applied Mathematics},
   publisher={Princeton University Press, Princeton, NJ},
        date={2008},
        ISBN={978-0-691-09679-7},
      review={\MR{2386879}},
}

\bib{HT15}{article}{
      author={Hirschfeld, J. W.~P.},
      author={Thas, J.~A.},
       title={Open problems in finite projective spaces},
    language={English},
        date={2015},
        ISSN={1071-5797},
     journal={Finite Fields Appl.},
      volume={32},
       pages={44\ndash 81},
}

\bib{Hub87}{article}{
      author={Huber, Michael},
       title={Baer cones in finite projective spaces},
        date={1987},
        ISSN={0047-2468},
     journal={J. Geom.},
      volume={28},
      number={2},
       pages={128\ndash 144},
         url={https://doi.org/10.1007/BF01221941},
      review={\MR{884978}},
}

\bib{HV90}{article}{
      author={Hefez, Abramo},
      author={Voloch, Jos\'{e}~Felipe},
       title={Frobenius nonclassical curves},
        date={1990},
        ISSN={0003-889X},
     journal={Arch. Math. (Basel)},
      volume={54},
      number={3},
       pages={263\ndash 273},
         url={https://doi.org/10.1007/BF01188523},
      review={\MR{1037617}},
}

\bib{KKPSSVW22}{article}{
      author={Kadyrsizova, Zhibek},
      author={Kenkel, Jennifer},
      author={Page, Janet},
      author={Singh, Jyoti},
      author={Smith, Karen~E.},
      author={Vraciu, Adela},
      author={Witt, Emily~E.},
       title={Lower bounds on the {$F$}-pure threshold and extremal
  singularities},
        date={2022},
        ISSN={2330-0000},
     journal={Trans. Amer. Math. Soc. Ser. B},
      volume={9},
       pages={977\ndash 1005},
         url={https://doi.org/10.1090/btran/106},
      review={\MR{4498775}},
}

\bib{Kle86}{inproceedings}{
      author={Kleiman, Steven~L.},
       title={Tangency and duality},
        date={1986},
   booktitle={Proceedings of the 1984 {V}ancouver conference in algebraic
  geometry},
      series={CMS Conf. Proc.},
      volume={6},
   publisher={Amer. Math. Soc., Providence, RI},
       pages={163\ndash 225},
      review={\MR{846021}},
}

\bib{KS20}{book}{
      author={Kiss, Gy\"{o}rgy},
      author={Sz\H{o}nyi, Tam\'{a}s},
       title={Finite geometries},
   publisher={CRC Press, Boca Raton, FL},
        date={2020},
        ISBN={978-1-4987-2165-3},
      review={\MR{4242245}},
}

\bib{Lan02}{book}{
      author={Lang, Serge},
       title={Algebra},
     edition={third},
      series={Graduate Texts in Mathematics},
   publisher={Springer-Verlag, New York},
        date={2002},
      volume={211},
        ISBN={0-387-95385-X},
         url={https://doi.org/10.1007/978-1-4613-0041-0},
      review={\MR{1878556}},
}

\bib{LN96}{book}{
      author={Lidl, Rudolf},
      author={Niederreiter, Harald},
       title={Finite fields},
     edition={2},
      series={Encyclopedia of Mathematics and its Applications},
   publisher={Cambridge University Press},
        date={1996},
}

\bib{Mat89}{book}{
      author={Matsumura, Hideyuki},
       title={Commutative ring theory},
     edition={Second},
      series={Cambridge Studies in Advanced Mathematics},
   publisher={Cambridge University Press, Cambridge},
        date={1989},
      volume={8},
        ISBN={0-521-36764-6},
        note={Translated from the Japanese by M. Reid},
      review={\MR{1011461}},
}

\bib{Par86}{article}{
      author={Pardini, Rita},
       title={Some remarks on plane curves over fields of finite
  characteristic},
        date={1986},
        ISSN={0010-437X},
     journal={Compositio Math.},
      volume={60},
      number={1},
       pages={3\ndash 17},
         url={http://www.numdam.org/item?id=CM_1986__60_1_3_0},
      review={\MR{867952}},
}

\bib{Ser91}{incollection}{
      author={Serre, Jean-Pierre},
       title={Lettre \`a {M}. {T}sfasman},
        date={1991},
       pages={11, 351\ndash 353 (1992)},
        note={Journ\'{e}es Arithm\'{e}tiques, 1989 (Luminy, 1989)},
      review={\MR{1144337}},
}

\bib{Sor94}{article}{
      author={S{\o}rensen, Anders~Bj{\ae}rt},
       title={On the number of rational points on codimension-{$1$} algebraic
  sets in {${\bf P}^n({\bf F}_q)$}},
        date={1994},
        ISSN={0012-365X},
     journal={Discrete Math.},
      volume={135},
      number={1-3},
       pages={321\ndash 334},
  url={https://doi-org.ezproxy.library.ubc.ca/10.1016/0012-365X(93)E0009-S},
      review={\MR{1310889}},
}

\bib{SV86}{article}{
      author={St\"{o}hr, Karl-Otto},
      author={Voloch, Jos\'{e}~Felipe},
       title={Weierstrass points and curves over finite fields},
        date={1986},
        ISSN={0024-6115},
     journal={Proc. London Math. Soc. (3)},
      volume={52},
      number={1},
       pages={1\ndash 19},
         url={https://doi.org/10.1112/plms/s3-52.1.1},
      review={\MR{812443}},
}

\bib{Tir17}{article}{
      author={Tironi, Andrea~Luigi},
       title={Hypersurfaces achieving the {H}omma-{K}im bound},
        date={2017},
        ISSN={1071-5797},
     journal={Finite Fields Appl.},
      volume={48},
       pages={103\ndash 116},
         url={https://doi.org/10.1016/j.ffa.2017.06.015},
      review={\MR{3705737}},
}

\bib{Zak93}{book}{
      author={Zak, F.~L.},
       title={Tangents and secants of algebraic varieties},
      series={Translations of Mathematical Monographs},
   publisher={American Mathematical Society, Providence, RI},
        date={1993},
      volume={127},
        ISBN={0-8218-4585-3},
        note={Translated from the Russian manuscript by the author},
      review={\MR{1234494}},
}

\end{biblist}
\end{bibdiv}

\ContactInfo

\end{document}